\documentclass[12pt,leqno]{article}
\usepackage{amssymb,amsfonts,amsmath,amsthm,amscd,mathrsfs}
\usepackage{psfrag}

\usepackage{graphicx}
\def\R{{{{\rm l} \kern -.15em {\rm R}}}}
\def\N{{{{\rm l} \kern -.15em {\rm N}}}}
\def\E{{{{\rm l} \kern -.15em {\rm E}}}}
\def\P{{{{\rm l} \kern -.15em {\rm P}}}}
\def\D{{{{\rm l} \kern -.15em {\rm D}}}}
\def\L{{{{\rm l} \kern -.15em {\rm L}}}}
\def\Z{{{{\rm Z} \kern -.35em {\rm Z}}}}

\def\IE{{\mathbb E}}
\def\IP{{\mathbb P}}
\def\IR{{\mathbb R}}
\def\IN{{\mathbb N}}

\def\IZ{{\mathbb Z}}

\def\n{\noindent}
\def\dsl{\textstyle\sum\limits}

\def\dis{\displaystyle}
\def\o{\omega}
\def\fr{\mbox{\footnotesize $\dis\frac{1}{2}$}}
\def\frvier{\mbox{\footnotesize $\dis\frac{1}{4}$}}
\def\ov{\overline}

\def\f{\footnotesize}
\def\r{\rightarrow}
\def\point{{\mbox{\large $.$}}}

\def\wt{\widetilde}

\def\cA{{\cal A}}
\def\cB{{\cal B}}
\def\cC{{\cal C}}

\def\cT{{\cal T}}

\def\cI{{\cal I}}

\def\cK{{\cal K}}

\def\cF{{\cal F}}

\def\cS{{\cal S}}

\def\cV{{\cal V}}
\def\cW{{\cal W}}

\dimendef\dimen=0

\newtheorem{theorem}{Theorem}[section]
\newtheorem{lemma}[theorem]{Lemma}

\newtheorem{proposition}[theorem]{Proposition}
\newtheorem{remark}[theorem]{Remark}

\begin{document}

\noindent

~

\bigskip
\begin{center}
{\bf PERCOLATION FOR THE VACANT SET OF \\ RANDOM INTERLACEMENTS}
\end{center}

\begin{center}
Vladas Sidoravicius\footnote[1]{CWI, Kruislaan 413, NL-1098 SJ, Amsterdam, and IMPA, Estrada Dona Castorina 110, Jardim 

 ~~Botanico, CEP 22460-320, Rio de Janeiro, RJ, Brasil.} and Alain-Sol Sznitman\footnote[2]{Departement Mathematik, ETH Z\"urich, CH-8092 Z\"urich, Switzerland.
 
Vladas Sidoravicius would like to thank the FIM for financial support and hospitality during his visits 

to ETH. His research was also partially supported by CNPq and FAPERJ.}
\end{center}

\begin{center}
\end{center}

\bigskip
\begin{abstract}
We investigate random interlacements on $\IZ^d$, $d \ge 3$. This model recently introduced in \cite{Szni07a} corresponds to a Poisson cloud on the space of doubly infinite trajectories modulo time-shift tending to infinity at positive and negative infinite times. A non-negative parameter $u$ measures how many trajectories enter the picture. Our main interest lies in the percolative properties of the vacant set left by random interlacements at level $u$. We show that for all $d \ge 3$ the vacant set at level $u$ percolates when $u$ is small. This solves an open problem of \cite{Szni07a}, where this fact has only been established when $d \ge 7$. It also completes the proof of the non-degeneracy in all dimensions $d \ge 3$ of the critical parameter $u_*$ of \cite{Szni07a}.
\end{abstract}

\vfill
\n

\vfill


\newpage
\thispagestyle{empty}
~

\newpage
\setcounter{page}{1}

 \setcounter{section}{-1}
 
 \section{Introduction}
 \setcounter{equation}{0}
 
The present work investigates the model of random interlacements on $\IZ^d$, $d \ge 3$, introduced in \cite{Szni07a}. Informally this is a translation invariant model which describes the microscopic structure left in the bulk by a random walk on a large discrete torus or on a discrete cylinder with base a large discrete torus when the walk is run up to times proportional to the number of sites in the torus or to the square of the number of sites in the base, cf.~\cite{Szni07b}, \cite{Wind08}. The main purpose of this article is to answer an open question of \cite{Szni07a} and show that for small $u > 0$ the vacant set at level $u$ in $\IZ^d$ left by random interlacements does percolate. In \cite{Szni07a} this had only been proved to be the case when $d \ge 7$. The present work shows that small dimensions behave in a similar fashion.

\medskip
We now describe the model somewhat informally and refer to Section 1 for precise definitions. Random interlacements consist of a cloud of paths constituting a Poisson point process on the space of doubly infinite $\IZ^d$-valued trajectories modulo time-shift tending to infinity at positive and negative infinite times. A non-negative parameter $u$ plays in essence the role of a multiplicative factor of the intensity measure of this Poisson point process. In a standard fashion one constructs on the same space $(\Omega, \cA, P)$, see below (\ref{1.9}), the whole family $\cI^u$, $u \ge 0$, of random interlacements at level $u \ge 0$, cf.~(\ref{1.16}). They are the traces on $\IZ^d$ of the cloud of trajectories modulo time-shift ``up to level $u$''. The random subsets $\cI^u$ increase with $u$ and for $u > 0$ are infinite random connected subsets of $\IZ^d$, ergodic under space translations, cf.~Theorem 2.1 of \cite{Szni07a}. The complement of $\cI^u$ in $\IZ^d$ is denoted with $\cV^u$. It is the so-called vacant set at level $u$, cf.~(\ref{1.17}). As shown in (2.16) of \cite{Szni07a}, the law $Q_u$ on $\{0,1\}^{\IZ^d}$ of the indicator function of $\cV^u$ is characterized by the property:
\begin{equation}\label{0.1}
Q_u[Y_x = 1, \;\mbox{for all} \;x \in K] = \exp\{ - u \;{\rm cap}(K)\}, \;\mbox{for all finite sets $K \subseteq \IZ^d$}\,,
\end{equation}

\n
where $Y_x$, $x \in \IZ^d$, stand for the canonical coordinates on $\{0,1\}^{\IZ^d}$ and cap$(K)$ for the capacity of $K$, cf.~(\ref{1.5}).

\medskip
Our main focus here lies in the percolative properties of $\cV^u$. For this purpose it is convenient to consider the non-increasing function
\begin{equation}\label{0.2}
\mbox{$\eta(u) = \IP[0$ belongs to an infinite connected component of $\cV^u], \;u \ge 0$}\,.
\end{equation}

\n
With Corollary 2.3 of \cite{Szni07a} one knows that the $\IP$-almost sure presence and absence of an infinite connected component (i.e.~``infinite cluster'') in $\cV^u$ are respectively equivalent to $\eta(u) > 0$ and $\eta(u) = 0$. One then introduces the critical parameter
\begin{equation}\label{0.3}
u_* = \inf\{u \ge 0; \;\eta(u) = 0\} \in [0,\infty]\,.
\end{equation}

\n
The main results of \cite{Szni07a}, cf.~Theorems 3.5 and 4.3, show that
\begin{equation}\label{0.4}
\mbox{$u_* < \infty$, for $d \ge 3$, and $u_* > 0$, for $d \ge 7$}\,,
\end{equation}

\n
i.e. $\cV^u$ does not percolate for large $u$, and at least when $d \ge 7$, percolates for small $u$. The main result in the present article, cf. Theorem 3.4, shows that
\begin{equation}\label{0.5}
u_* > 0, \;\mbox{for $d \ge 3$}\,,
\end{equation}

\n
and even that $\cV^u$ percolates in planes for small $u$. This solves an open problem of \cite{Szni07a} and proves that $u_*$ is non-degenerate in all dimensions. Let us also mention that with Theorem 1.1 and Corollary 1.2 of \cite{Teix08}, it is known when $\eta(u) > 0$, i.e. when $\cV^u$ percolates, the infinite cluster is almost surely unique, that $\eta$ is continuous on $[0,u_*]$ and has at most one point of discontinuity at $u_*$. It is at present unknown whether $\cV^{u_*}$ percolates or not.

\medskip
Let us give some comments on the proof of (\ref{0.5}). The difficulty in proving (\ref{0.5}) stems from the fact that the usual Peierls-type arguments that require a good enough exponential bound on $\IP[\cI^u \supseteq A]$ in terms of the cardinality $|A|$ for $A$ finite in $\IZ^2$ (viewed as a subset of $\IZ^d)$, so far only work when $d \ge 18$, see Remark 2.5 3) of \cite{Szni07a}. In fact when $d = 3$, there is no such exponential bound, cf.~(\ref{1.20}). This difficulty is closely related to the long range dependence present in the model: as shown in (1.68) of \cite{Szni07a}, the correlation of the events $\{x \in \cV^u\}$ and $\{y \in \cV^u\}$ decays as $c(u) |x-y|^{-(d-2)}$, when $|x-y|$ tends to infinity. To bypass this obstruction we employ a renormalization technique which is different but has a similar flavor to the methods of Section 3 of \cite{Szni07a}. To prove (\ref{0.5}) we show, cf.~(\ref{3.25}), (\ref{3.26}), that for small $u$ the probability that a $*$-circuit of $\cI^u \cap \IZ^2$ surrounds the origin is smaller than $1¤$, (we refer to the beginning of Section 1 for the definition of $*$-paths). For this purpose we develop estimates showing that for small $u$ the presence of long $*$-paths in $\cI^u \cap \IZ^2$ is unlikely. We consider a sequence of functions
\begin{align}
q_n(u) \mbox{``$=$''} &\;  \mbox{$\IP$-probability that  $\cV^u \cap \IZ^2$ contains a $*$-path from a given block}\label{0.6}
\\
&\; \mbox{of size $L_n$ to the complement of its $L_n$-neighborhood, for $n \ge 0$,}\nonumber
\end{align}

\n
(we refer to (\ref{2.7}), (\ref{2.8}) for the precise expression), with the aim of proving that for small $u$, $q_n(u)$ decays with $n$ at least as an inverse power of $L_n$. The sequence of length scales $L_n$, $n \ge 0$, in (\ref{0.6}) grows rapidly, and see (\ref{2.1}), (\ref{2.2}):
\begin{equation}\label{0.7}
L_n \approx L_0^{(1 + a)^n}, \;n \ge 0, \;\mbox{with $a = \mbox{\f $\dis\frac{1}{100}$}$}\,.
\end{equation}

\n
We derive a recurrence relation bounding $q_{n+1}(u_{n+1})$ in terms of $q_n(u_n)$ along a decreasing sequence $u_n$ such that, cf.~(\ref{2.67}),
\begin{equation}\label{0.8}
u_{n+1} = \Big( 1 + \mbox{\f $\dis\frac{1}{\log L_n}$}\Big)^{-1} \;u_n, \;\mbox{for $n \ge 0$} \,.
\end{equation}

\n
As a result of (\ref{0.7}) this sequence converges to a positive value $u_\infty > 0$. The recurrence relation is based on Proposition 2.1 and hinges on the ``sprinkling technique'' of \cite{Szni07a}, where more independent paths are thrown in, with the purpose of dominating long range interactions present in the model. In the proof of Theorem 4.3 of \cite{Szni07a}, when showing $u_* > 0$, for $d \ge 7$, these long range interactions could be bounded in a rather primitive way, with not too dire consequences thanks to the assumption $d \ge 7$. An important contribution of the present work is that we are able to control these interactions even in the case of small dimension, see also Remark \ref{rem2.3} 2). The result of the renormalization scheme, cf.~Theorem \ref{theo2.5}, is that for suitable dimension dependent constants $c, c^\prime, c^{\prime\prime}$, if we can find
\begin{equation}\label{0.9}
\mbox{$L_0 \ge c$ and $u_0 \ge c^\prime \;\mbox{\f $\dis\frac{(\log L_0)^2}{L_0^{d-2}}$}$ such that $q_0 (u_0) \le c^{\prime\prime} \,L_0^{-(1 + 2a)}$}\,,
\end{equation}
then
\begin{equation}\label{0.10}
\mbox{for all $n \ge 0$, $q_n(u_n) \le c^{\prime\prime} \,L_n^{-(1 + 2a)}$}\,.
\end{equation}

\medskip\n
This procedure essentially reduces the proof of (\ref{0.5}) to checking (\ref{0.9}). This step is carried out in Theorem 3.1 where it is shown that
\begin{equation}\label{0.11}
\mbox{$\lim\limits_{L_0 \rightarrow \infty} \;L_0^\rho \;q_0(u_0) = 0$, for all $\rho > 0$, with $u_0 = \mbox{\f $\dis\frac{c^\prime(\log L_0)^2}{L_0^{d-2}}$}$} \;.
\end{equation}

\n
The two-dimensional character of the event in the right-hand side of (\ref{0.6}) plays here a crucial role. Replacing $\IZ^2$ with $\IZ^d$ would still lead to a rather similar recurrence relation between $q_{n+1}(u_{n+1})$ and $q_n(u_n)$. However one could not initiate the induction in this modified set-up, cf.~Remark \ref{rem2.6}, (and (\ref{0.11}) would be replaced with $\lim_{L_0 \r \infty} q_0(u_0) = 1$). Interestingly the proof of (\ref{0.11}) relies on arguments reminiscent of some of the steps that appear in the derivation of lower bounds on the disconnection times of discrete cylinders by random walks, see Section 2 of \cite{DembSzni06} or Section 5 of \cite{Szni08}.

\medskip
We will now describe the organization of this article.

\medskip
In Section 1 we introduce some notation and recall useful facts concerning random interlacements.

\medskip
In Section 2 we develop the renormalization scheme. The induction step is carried out in Proposition 2.1. The application of the induction step to the proof of the fact that (\ref{0.10}) is a consequence of (\ref{0.9}) appears in Proposition 2.4 and Theorem 2.5.

\medskip
In Section 3 we prove (\ref{0.11}) in Theorem \ref{theo3.1}. This enables to initiate the induction and yields (\ref{0.10}) for a decreasing sequence with positive limit $u_\infty$. As a consequence we show in Theorem \ref{theo3.4} that for small $u > 0$, $\IP$-almost surely $\cV^u \cap \IZ^2$ percolates, which in particular yields (\ref{0.5}).

\medskip
Finally let us explain the convention we use for constants. Throughout the text $c$ or $c^\prime$ denote positive constants which solely depend on $d$, with values changing from place to place. The numbered constants $c_0,c_1,\dots$ are fixed and refer to the value at their first appearance in the text. Dependence of constants on additional parameters appears in the notation.

\section{Notation and some facts about random interlacements}
\setcounter{equation}{0}

The main purpose of this section is to introduce additional notation and recall various useful facts concerning random interlacements.

\medskip
We let $|\cdot|$ and $|\cdot|_\infty$ respectively stand for the Euclidean and the $\ell^\infty$-distance on $\IZ^d$. Unless explicitly mentioned we assume $d \ge 3$ throughout the article. We say that $x,y$ in $\IZ^d$ are neighbors, respectively $*$-neighbors, if $|x - y| = 1$, respectively $|x - y|_\infty = 1$. By finite path, respectively finite $*$-path, we mean a sequence $x_0,x_1,\dots,x_N$ in $\IZ^d$, $N \ge 0$, such that $x_i$ and $x_{i+1}$ are neighbors, respectively $*$-neighbors, for each $0 \le i < N$. We also sometimes write path, or $*$-path, in place of finite path, or finite $*$-path, when this causes no confusion. With $B(x,r)$ and $S(x,r)$ we denote the closed $|\cdot |_\infty$-ball and $|\cdot |_\infty$-sphere with radius $r \ge 0$ and center $x \in \IZ^d$. For $A,B$ subsets of $\IZ^d$ we write $A + B$ for the set of elements $x + y$ with $x$ in $A$ and $y$ in $B$, and $d(A,B) = \inf\{|x-y|_\infty$; $x \in A, y \in B\}$, for the mutual $\ell^\infty$-distance between $A$ and $B$. When $A$ is a singleton $\{x\}$, we write $d(x,B)$ for simplicity. The notation $U \subset \subset \IZ^d$ means that $U$ is a finite subset of $\IZ^d$. Given $U$ subset of $\IZ^d$ we denote with $|U|$ the cardinality of $U$, with $\partial U$ the boundary of $U$ and $\partial_{\rm int} U$ the interior boundary of $U$:
\begin{equation}\label{1.1}
\partial U = \{x \in U^c; \exists y \in U, \,|x-y| = 1\}, \;\;\partial_{\rm int} U = \{x \in U; \exists y \in U^c, \,|x-y| = 1\}\,.
\end{equation}

\n
The canonical basis of $\IR^d$ is denoted with $(e_i)_{1 \le i \le d}$, and we tacitly identify $\IZ^2$ with $\IZ e_1 + \IZ e_2 \subseteq \IZ^d$.

\medskip
We write $W_+$ for the space of nearest neighbor $\IZ^d$-valued trajectories defined for non-negative times and tending to infinity. We denote with $\cW_+$, $X_n, n \ge 0$, and $\cF_n, n \ge 0$, the canonical $\sigma$-algebra, the canonical process and canonical filtration on $W_+$. We let $\theta_n, n \ge 0$, stand for the canonical shift on $W_+$. Since $d \ge 3$, simple random walk on $\IZ^d$ is transient and for $x \in \IZ^d$ we denote with $P_x$ the restriction of the canonical law of simple random walk starting at $x$ to the set $W_+$, which has full measure. We write $E_x$ for the corresponding expectation. We also define $P_\rho = \sum_{x \in \IZ^d} \rho(x) \,P_x$  when $\rho$ is a measure on $\IZ^d$, and write $E_\rho$ for the corresponding expectation. Given $U \subseteq \IZ^d$, we let $H_U, \wt{H}_U, T_U$ stand for the respective entrance time, hitting time of $U$, and exit time from $U$:
\begin{equation}\label{1.2}
\begin{split}
H_U & = \inf\{n \ge 0; \;X_n \in U\}, \;\wt{H}_U = \inf\{n \ge 1; X_n \in U\},  \;\mbox{and}
\\[1ex]
T_U & = \inf\{n \ge 0; \,X_n \notin U\}\,.
\end{split}
\end{equation}

\n
In case of a singleton $U = \{x\}$, we write $H_x$ or $\wt{H}_x$ for simplicity.

\medskip
We denote with $g(\cdot, \cdot)$ the Green function of the walk
\begin{equation}\label{1.3}
g(x,y) = \dsl_{n \ge 0} \;P_x [X_n = y], \;x, y \in \IZ^d\,,
\end{equation}

\n
which is symmetric in its two variables, and $g(y) = g(0,y)$ so that $g(x,y) = g(y-x)$, due to translation invariance. Given $K \subset \subset \IZ^d$ we write $e_K$ for the equilibrium measure of $K$ and cap$(K)$ for the capacity of $K$, so that:
\begin{align}
e_K(x) & = P_x [\wt{H}_K = \infty], \;\mbox{for $x \in K$}\,,\label{1.4}
\\[1ex]
& = 0, \; \mbox{for $x \notin K$, and}\nonumber
\\[1ex]
{\rm cap}(K) & = e_K(\IZ^d) = \dsl_{x \in K} \,P_x [\wt{H}_K = \infty]\,. \label{1.5}
\end{align}

\n
It is straightforward to see from (\ref{1.5}) that the capacity is subadditive in the sense that ${\rm cap}(K \cup K^\prime) \le {\rm cap}(K) + {\rm cap}(K^\prime)$ for $K, K^\prime$ finite subsets of $\IZ^d$. Further the probability to enter $K$ can be expressed as
\begin{equation}\label{1.6}
P_x [H_K < \infty] = \dsl_{y \in K} g(x,y) \,e_K(y), \;\mbox{for} \;x \in \IZ^d\,.
\end{equation}

\n
One also has the bounds, (see for instance (\ref{1.9}) of \cite{Szni07a}):
\begin{equation}\label{1.7}
\dsl_{y \in K} g(x,y) / \sup\limits_{z \in K} \,\Big(\dsl_{y \in K} g(z,y)\Big) \le P_x [H_K < \infty] \le \dsl_{y \in K} g(x,y) / \inf\limits_{z \in K} \,\Big(\dsl_{y \in K} g(z,y)\Big)\,,
\end{equation}

\n
for $x$ in $\IZ^d$, from which can infer with the help of classical bounds on the Green function, cf.~\cite{Lawl91}, p.~31, that
\begin{equation}\label{1.8}
c\,L^{d-2} \le {\rm cap}\big(B(0,L)\big) \le c^\prime \, L^{d-2}, \;\mbox{for} \;L \ge 1\,.
\end{equation}

\n
To introduce random interlacements we need some further objects. We denote with $W$ the space of doubly infinite nearest neighbor $\IZ^d$-valued trajectories, which tend to infinity at positive and negative infinite times, and with $W^*$ the space of equivalence classes of trajectories in $W$ modulo time-shift. The canonical projection from $W$ onto $W^*$ is denoted by $\pi^*$. We endow $W$ with its canonical $\sigma$-algebra generated by the canonical coordinates $X_n, n \in \IZ$, and $W^*$ with $\cW^* = \{A \subseteq W^*$; $(\pi^*)^{-1} (A) \in \cW\}$, the largest $\sigma$-algebra on $W^*$ for which $\pi^*: (W, \cW) \rightarrow (W^*, \cW^*)$ is measurable.

\medskip
We will now describe the space $(\Omega, \cA, \IP)$ where random interlacements are defined. We consider the space of point measures on $W^* \times \IR_+$: 
\begin{align}
\Omega  = \Big\{ &\omega  = \dsl_{i \ge 0} \,\delta_{(w_i^*, u_i)}, \;\mbox{with $(w_i^*, u_i) \in W^* \times \IR_+$, for $i \ge 0$, and}\label{1.9}
\\[-0.5ex]
&\o (W^*_K \times [0,u]) < \infty, \;\mbox{for any} \;K \subset \subset \IZ^d, \,u \ge 0\Big\}\,, \nonumber
\end{align}

\n
where for $K \subset \subset \IZ^d$, $W^*_K \subseteq W^*$ is the set of trajectories modulo times-shift which enter $K$:
\begin{equation}
\mbox{$W^*_K = \pi^* (W_K)$, and $W_K = \{w \in W$, for some $n \in \IZ$, $X_n (w) \in K\}$}\,.
\end{equation}

\n
We endow $\Omega$ with the $\sigma$-algebra $\cA$ generated by the evaluation maps $\o \rightarrow \o(D)$, where $D$ runs over the product $\sigma$-algebra $\cW^* \times \cB(\IR_+)$. We denote with $\IP$ the probability on $(\Omega, \cA)$, which is the Poisson point measure with intensity $\nu(d w^*) du$, giving finite mass to the sets $W^*_K \times [0, u]$, for $K \subset \subset \IZ^d$, $u \ge 0$, where $\nu$ is the unique $\sigma$-finite measure on $(W^*, \cW^*)$ such that for any $K \subset \subset \IZ^d$, cf.~Theorem 1.1. of \cite{Szni07a}:
\begin{equation}\label{1.11}
1_{W^*_K} \,\nu = \pi^* \circ Q_K\,,
\end{equation}

\n
with $Q_K$ the finite measure on $W^0_K$, the subset of $W_K$ of trajectories which enter $K$ for the first time at time $0$, such that for $A, B \in \cW_+$, $x \in \IZ^d$, (see (\ref{1.4}) for the notation):
\begin{equation}\label{1.12}
Q_K [(X_{-n})_{n \ge 0} \in A, \;X_0 = x, \;(X_n)_{n \ge 0} \in B] = P_x [A \,|\,\wt{H}_K = \infty] \;e_K(x) \,P_x[B]\,.
\end{equation}

\n
Given $K \subset \subset \IZ^d$, $u \ge 0$, one further defines on $(\Omega, \cA)$ the random point process with values in the set of finite point measures on  $(W_+, \cW_+)$:
\begin{equation}\label{1.13}
\mu_{K,u}(\o) = \dsl_{i \ge 0} \,\delta_{(w_i^*)^{K,+}}  1_{\{w_i^* \in W^*_K, u_i \le u\}}, \;\mbox{for} \;\o = \dsl_{i \ge 0}\,\delta_{(w_i^*,u_i)}\,,
\end{equation}

\n
where $(w^*)^{K,+}$ stands for the trajectory in $W_+$ which follows step by step $w^* \in W^*_K$ from the time it first enters $K$. One then knows from Proposition 1.3 of \cite{Szni07a} that for $K \subset \subset \IZ^d$:
\begin{equation}\label{1.14}
\mbox{$\mu_{K,u}$ is a Poisson point process on $(W_+,\cW_+)$ with intensity measure $u\,P_{e_K}$}\,,
\end{equation}

\n
where the notation has been introduced above (\ref{1.2}). When $0 \le u^\prime < u$, and $K \subset \subset \IZ^d$, one can define $\mu_{K, u^\prime, u}(\o)$ for $\o \in \Omega$, in analogy to (\ref{1.13}), simply replacing the inequality $u_i \le u$, by the condition $u^\prime < u_i \le u$ in the formula for $\mu_{K,u}(\o)$. Once then finds that for $0 \le u^\prime < u$ and $K \subset \subset \IZ^d$:
\begin{equation}\label{1.15}
\begin{split}
& \mbox{$\mu_{K,u^\prime,u}$ and $\mu_{K,u^\prime}$ are independent Poisson point processes on $(W_+, \cW_+)$}
\\
&\mbox{with respective intensity measures $(u - u^\prime) \,P_{e_K}$ and $u^\prime \,P_{e_K}$}\,.
\end{split}
\end{equation}

\n
Given $\o \in \Omega$, the interlacement at level $u \ge 0$, is the subset of $\IZ^d$:
\begin{equation}\label{1.16}
\begin{split}
\cI^u(\o) & = \bigcup\limits_{u_i \le u} \;{\rm range} \;(w_i^*), \;\mbox{if} \;\o = \dsl_{i \ge 0} \,\delta_{(w_i^*,u_i)}\,,
\\
& = \bigcup\limits_{K \subset \subset \IZ^d} \; \bigcup\limits_{w \in {\rm Supp} \,\mu_{K,u}(\o)} w(\IN)\,,
\end{split}
\end{equation}

\n
where for $w^* \in W^*$, range $(w^*) = w(\IZ)$, for any $w \in W$, with $\pi^*(w) = w^*$, and the notation Supp $\mu_{K,u}(\o)$ refers to the support of the point measure $\mu_{K,u}(\o)$. The vacant set at level $u$ is then defined as
\begin{equation}\label{1.17}
\cV^u(\o) = \IZ^d \backslash \cI^u(\o), \;\mbox{for} \;\o \in \Omega, \;u \ge 0\,.
\end{equation}
One then finds, see (1.54) of \cite{Szni07a}, that
\begin{equation}\label{1.18}
\mbox{$\cI^u(\o) \cap K = \bigcup\limits_{w \in {\rm Supp} \,\mu_{K,u}(\o)} w(\IN) \cap K$, for $K \subset \subset \IZ^d$, $u \ge 0$, $\o \in \Omega$}\,,
\end{equation}

\n
and it follows with (\ref{1.14}) that for $u \ge 0$,
\begin{equation}\label{1.19}
\IP[\cV^u \supseteq K] = \exp\{- u \,{\rm cap} (K)\}, \;\mbox{for all} \;K \subset \subset \IZ^d\,,
\end{equation}

\n
a property which leads to the characterization by (\ref{0.1}) of the law $Q_u$ on $\{0,1\}^{\IZ^d}$ of the random subset $\cV^u$ of $\IZ^d$, see Remark 2.2 2) of \cite{Szni07a}. As mentioned in the introduction $Q_u$ is ergodic under spatial translations, cf.~Theorem 2.1 of \cite{Szni07a}, and for $u > 0$, $\cI^u(\o)$ is $\IP$-a.s. an infinite connected subset of $\IZ^d$, cf.~Corollary 2.3 of \cite{Szni07a}.  Intuitively it can be thought of as a ``random fabric''.

\begin{remark}\label{rem1.1} \rm
Since our principal objective is to prove that when $u > 0$ is small $\cV^u$ percolates, it is important to point out that $Q_u$ does not dominate any product of non-degenerate iid Bernoulli variables. Indeed one knows from Remark 2.5 2) of \cite{Szni07a} that for $u > 0$, and $L \ge c(u)$, (see below (\ref{0.1}) for the notation):
\begin{equation}\label{1.20}
\IP[\cI^u \supseteq B(0,L)] = Q_u [Y_x = 0, \;\mbox{for all} \;x \in B(0,L)] \ge c\,\exp\{- c \,L^{d-2} \log L\}\,.
\end{equation}

\n
This shows that the probability that $\cV^u \cap B(0,L)  = \phi $ is rather ``fat''. In particular the law $Q_u$ of the indicator function of $\cV^u$ cannot stochastically dominate, see \cite{Ligg85}, p.~71-74, the law of iid non-degenerate Bernoulli variables indexed by $\IZ^d$. Note that when $d = 3$, the same argument even proves that the law of the indicator function of $\cV^u \cap \IZ^2$ cannot stochastically dominate the law of iid non-degenerate Bernoulli variables indexed by $\IZ^2$. This rather high probability of absence of $\cV^u$ in large boxes makes it difficult to prove that $\cV^u$ percolates for small $u > 0$, with a strategy based on dynamical or static renormalization in the spirit of Chapter 7 of \cite{Grim99}.

\medskip
Incidentally note that with (\ref{0.1}) and (\ref{1.8}) one also finds
\begin{equation}\label{1.21}
\IP [\cV^u \supseteq B(0,L)] = Q_u [Y_x = 1, \;\mbox{for all} \;x\in B(0,L)] \ge \exp\{- c \,L^{d-2}\}\,,
\end{equation}

\n
and the probability that $\cV^u$ covers $B(0,L)$ is rather ``fat'' as well. Just as above we conclude that $Q_u$ is not stochastically dominated by the law of non-degenerate iid Bernoulli variables indexed by $\IZ^d$. \hfill $\square$
\end{remark}

\section{The induction step}
\setcounter{equation}{0}

In this section we develop a renormalization scheme which aims at showing that the sequence of probabilities $q_n(u_n)$ that $\cI^{u_n} \cap \IZ^2$ contains a $*$-path between a given block of side-length of order $L_n$ and the complement of its $L_n$-neighborhood, tends to zero. The sequence of length scales $L_n$, $n \ge 0$, grows rapidly to infinity, cf.~(\ref{2.2}), whereas $u_n$, $n \ge 0$, is a decreasing sequence of levels tending to $u_\infty > 0$, cf.~(\ref{2.67}). The heart of the matter in this section is the derivation of a recurrence relation enabling the control of $q_{n+1}(u_{n+1})$ in terms of $q_n(u_n)$, cf.~(\ref{2.65}). For this purpose we use a ``sprinkling technique'': trajectories of the interlacement with levels in $(u_{n+1}, u_n]$ are used to dominate the long range interactions in the problem and restore some independence, see (\ref{2.61}). The main result of this section is Theorem \ref{theo2.5}. It reduces the task of proving a quantitative decay to zero of the sequence $q_n(u_n)$ to the question of being able to initiate the recurrence, see (\ref{0.11}). Most of the work in the derivation of Theorem \ref{theo2.5} is carried out in Proposition \ref{prop2.1} and Proposition \ref{prop2.4}. We first need some notation.

\medskip
We introduce a sequence of length scales as follows. We set
\begin{equation}\label{2.1}
a = \mbox{\f $\dis\frac{1}{100}$} \,,
\end{equation}
and given $L_0 > 1$, define by induction
\begin{equation}\label{2.2}
L_{n+1} = \ell_n \,L_n, \;\mbox{for $n \ge 0$, where $\ell_n = 100 [L^a_n] + 1 \;( \ge L_n^a)$}\,.
\end{equation}

\medskip\n
We then introduce for each $n$ a collection of pairwise disjoint $d$-dimensional boxes covering $\IZ^2$, where we recall the convention introduced below (\ref{1.1}). The index set of labels of boxes at level $n$ is
\begin{equation}\label{2.3}
I_n = \{ m = (n,i); \;i \in \IZ^2\}, \;\mbox{for $n \ge 0$}\,,
\end{equation}

\medskip\n
and to each $m = (n,i) \in I_n$ we attach the $d$-dimensional boxes
\begin{align}
C_m & = \big([-L_n, L_n)^d + 2 L_n i\big) \cap \IZ^d \label{2.4}
\\[1ex]
\wt{C}_m & = \bigcup\limits_{m^\prime \in I_n: d (C_{m^\prime}, C_m) \le 1} C_{m^\prime} = \big([-3 L_n, 3 L_n\big)^d + 2 L_n i) \cap \IZ^d\,. \label{2.5}
\end{align}

\medskip\n
So for each $n \ge 0$, the boxes $C_m, m \in I_n$, are pairwise disjoint and their union covers $\IZ^2$. As for $\wt{C}_m$ it is the union of $C_m$ and the $\ell^\infty$-neighboring boxes of $C_m$ at level $n$. Note that one has the following ``hierarchical'' property: given $m \in I_{n+1}$, the trace of $C_m$ on $\IZ^2$ is partitioned by the respective traces on $\IZ^d$ of the boxes at level $n$ contained in $C_m$:
\begin{equation}\label{2.6}
C_m \cap \IZ^2 = \bigcup\limits_{m^\prime \in I_n:  C_{m^\prime} \subseteq C_m} \, C_{m^\prime} \cap \IZ^2, \;\;\mbox{for $n \ge 0$, and $m \in I_{n+1}$} \,.
\end{equation}

\n
Indeed using the fact that $\ell_n$ is odd, cf.~(\ref{2.2}), one simply writes 
\begin{equation*}
[-L_{n+1}, L_{n+1}) = \bigcup\limits_{k \;{\rm odd}; - \ell_n \le k < \ell_n} [k\,L_n, (k+2) \,L_n) = \bigcup\limits_{\ell \;{\rm even}; - \ell_n < \ell < \ell_n} \big([-L_n, L_n) + \ell \,L_n\big)\,,
\end{equation*}

\medskip\n
and inserts this identity ``coordinatewise'' into (\ref{2.4}). Given $u \ge 0$, $n \ge 0$ and $m \in I_n$, one defines the event
\begin{equation}\label{2.7}
\mbox{$B^u_m =$ there is a $*$-path from $C_m$ to $\partial_{\rm int} \,\wt{C}_m$ in $\cI^u \cap \IZ^2$}\,,
\end{equation}

\medskip\n
(we refer to the beginning of Section 1 for the notation). One also defines the probability:
\begin{equation}\label{2.8}
\mbox{$q_n(u) = \IP[B^u_m]$, where $m \in I_n$ is arbitrary}\,,
\end{equation}

\n
and we have used the translation invariance of $Q_u$, see (\ref{0.1}).

\medskip
As already mentioned we aim at deriving a bound of $q_{n+1} (u_{n+1})$ in terms of $q_n(u_n)$, along certain decreasing sequences $u_n$ satisfying $u_\infty = \lim_n \,u_n > 0$. With this objective in mind, it is convenient to introduce for $n \ge 0$, the collections $\cK_1$ and $\cK_2$ of labels of boxes at level $n$ contained in $\wt{C}_m$, where $m = (n + 1, 0) \in I_{n+1}$, (i.e. $C_m$ is the box at level $n+1$ containing the origin):
\begin{align}
\cK_1 & = \Big\{m^\prime \in I_n; \;m^\prime = (n, i^\prime), \;\mbox{where $i^\prime = (i^\prime_1, i^\prime_2)$ with max $(|i^\prime_1|, |i^\prime_2|) = \mbox{\f $\dis\frac{\ell_n - 1}{2}$}\Big\}$} \label{2.9}
\\[-0.5ex]
& = \{m^\prime \in I_n; \;C_{m^\prime} \cap \partial_{\rm int} \;C_m \not= \phi\}\,,\nonumber
\\[2ex]
\cK_2 & = \{m^\prime \in I_n; \;m^\prime = (n, i^\prime), \;\mbox{where $i^\prime = (i^\prime_1, i^\prime_2)$ with max $(|i^\prime_1|, |i^\prime_2|) =\ell_n \}$} \label{2.10}
\\
& = \{m^\prime \in I_n; \;C_{{m^\prime}} \cap S(0,2\, L_{n+1}) \not= \phi\}\,.\nonumber
\end{align}

\begin{center}
\psfragscanon
\includegraphics{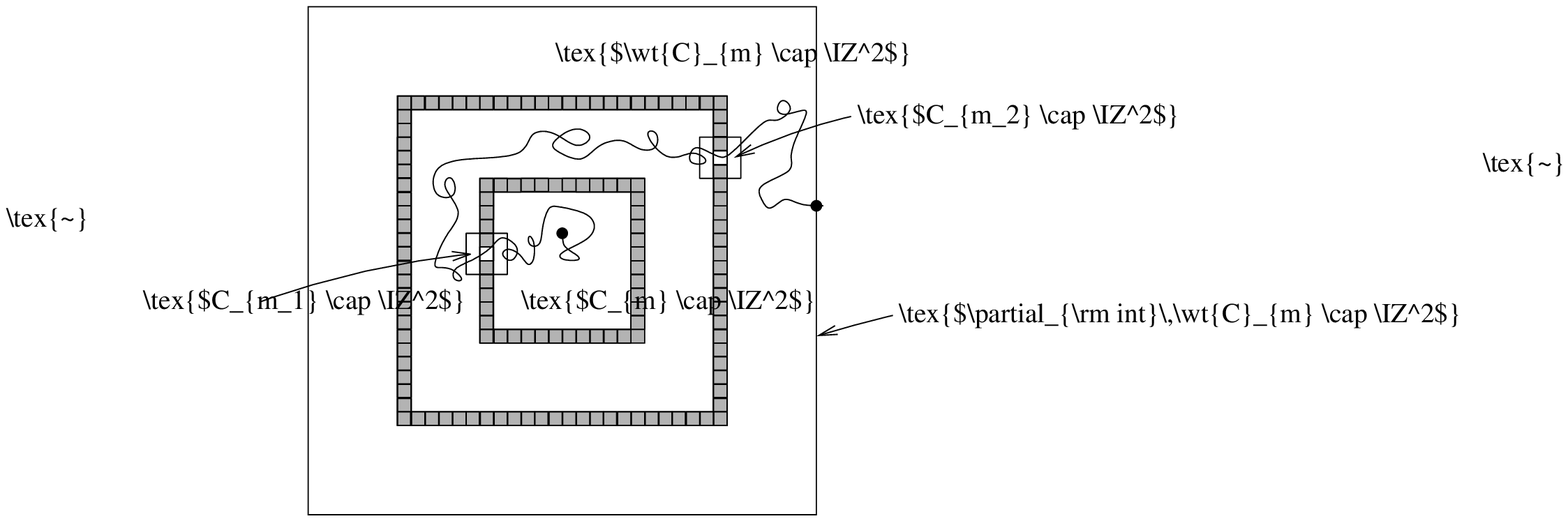}{1000}
\end{center}

\begin{center}
\begin{tabular}{ll}
Fig. 1: & An illustration of the event $B^u_m$ with a\\
&$*$-path in $\cI^u \cap \IZ^2$ from $C_m$ to $\partial_{\rm int} \,\wt{C}_m$.
\end{tabular}
\end{center}

\medskip\n
A key ingredient of the renormalization scheme comes from the following 

\begin{proposition}\label{prop2.1} $(d \ge 3)$

\medskip
For $L_0 \ge c$, for all $n \ge 0$ and
\begin{equation}\label{2.11}
0 < u^\prime \le \Big(1 + \dis\frac{c_0}{\ell_n^{d-2}}\Big)^{-1} \,u\,,
\end{equation}
one has
\begin{equation}\label{2.12}
q_{n+1}(u^\prime) \le c_1\,\ell_n^2 \big(q_n(u)^2 + u^\prime \,L_n^{-2} + e^{-c_2(u - u^\prime)L_n^{d-2}}\big)\,.
\end{equation}
\end{proposition}

\begin{proof}
We consider $n \ge 0$, $0 < u^\prime < u$ and $m = (n+1,0) \in I_{n+1}$, as above (\ref{2.9}). Observe that any $*$-path in $\IZ^2$ from $C_m$ to $\partial_{\rm int} \,\wt{C}_m$ necessarily meets some $C_{m_1}$ with $m_1$ in $\cK_1$ and some $C_{m_2}$ with $m_2$ in $\cK_2$, and is neither contained in $\wt{C}_{m_1}$ nor $\wt{C}_{m_2}$, so that with a rough counting argument
\begin{equation}\label{2.13}
q_{n+1} (u^\prime) \le c\,\ell^2_n \,\sup\limits_{m_1,m_2} \,\IP[B^{u^\prime}_{m_1} \cap B^{u^\prime}_{m_2}]\,,
\end{equation}

\n
where the supremum runs over $m_1$ in $\cK_1$ and $m_2$ in $\cK_2$. Given such $m_1$ and $m_2$ we write
\begin{equation}\label{2.14}
V = \wt{C}_{m_1} \cup \wt{C}_{m_2}\,,
\end{equation}

\medskip\n
and introduce the decomposition (in the notation of (\ref{1.13})):
\begin{equation}\label{2.15}
\mu_{V,u}  = \mu_{1,1} + \mu_{1,2} + \mu_{2,1} + \mu_{2,2}\,,
\end{equation}
where for $i,j$ distinct in $\{1,2\}$ we have set
\begin{equation}\label{2.16}
\begin{split}
\mu_{i,j} & = 1_{\{X_0 \in \wt{C}_{m_i}, \;H_{\wt{C}_{m_j}} < \infty\}}\,\mu_{V,u}\,,
\\[1ex]
\mu_{i,i} & = 1_{\{X_0 \in \wt{C}_{m_i}, \;H_{\wt{C}_{m_j}} = \infty\}}\,\mu_{V,u}\,.
\end{split}
\end{equation}

\n
Similarly when $\mu_{V,u^\prime}$ and $\mu_{V,u^\prime,u}$, (see above (\ref{1.15})), play the role of $\mu_{V,u}$ one obtains the identities (with hopefully obvious notations):
\begin{align}
\mu_{V,u^\prime} & = \mu^\prime_{1,1} + \mu^\prime_{1,2} + \mu^\prime_{2,1} + \mu^\prime_{2,2}\,,\label{2.17}
\\[1ex]
\mu_{V,u^\prime,u} & = \mu^*_{1,1} + \mu^*_{1,2} + \mu^*_{2,1} + \mu^*_{2,2}\,,\label{2.18}
\end{align}
together with
\begin{equation}\label{2.19}
\mu_{i,j} = \mu^\prime_{i,j} + \mu^*_{i,j}, \;\;\mbox{for} \;1 \le i,j \le 2\,.
\end{equation}

\n
In view of (\ref{1.15}) we also see that
\begin{equation}\label{2.20}
\mbox{$\mu^\prime_{i,j}$, $\mu^*_{i,j}$, $1 \le i, j \le 2$, are independent Poisson point processes on $W_+$}\,.
\end{equation}

\n
Roughly speaking the possible dependence between $B^{u^\prime}_{m_1}$ and $B^{u^\prime}_{m_2}$ in the right-hand side of (\ref{2.13}) originates from the contribution of $\mu^\prime_{1,2} + \mu^\prime_{2,1}$, see (\ref{2.17}), when one considers the trace on $V \cap \IZ^2$ of the trajectories in the support of $\mu_{V,u^\prime}$. In essence our strategy is to exhibit some ``domination'' of the trace on $V \cap \IZ^2$ of trajectories in the support of $\mu^\prime_{1,2} + \mu^\prime_{2,1}$ in terms of the corresponding trace of trajectories in the support of $\mu^*_{1,1} + \mu^*_{2,2}$ and  ``correction terms'', when $u$ is sufficiently bigger than $u^\prime$. This is the sprinkling technique and it will produce a decoupling and a natural control of $q_{n+1}(u^\prime)$ in terms of $q_n(u)^2$ thanks to (\ref{2.19}), (\ref{2.20}).

\medskip
We now introduce the $\ell^\infty$-neighborhood of size $[\frac{L_{n+1}}{10}]$ of $V = \wt{C}_{m_1} \cup \wt{C}_{m_2}$:
\begin{equation}\label{2.21}
U = \Big\{z \in \IZ^d; \;d(z,V) \le \mbox{\f $\dis\frac{L_{n+1}}{10}$}\Big\}\,.
\end{equation}

\medskip\n
This set is the union of two disjoint boxes which are translates of $[-R,R)^d$, with $R = 3 L_n + [\frac{L_{n+1}}{10}]$, and the $\ell^\infty$-norm of the vector translating one to the other box is at least $L_{n+1}$.

\medskip
For a trajectory in $W_+$, the times $R_k$, $k \ge 1$, of successive returns to $V$, and $D_k$, $k \ge 1$, of successive departures from $U$ are defined as, (see (\ref{1.2}) for the notation):
\begin{equation}\label{2.22}
\begin{split}
R_1 & = H_V, \,D_1 = T_U \circ \theta_{R_1} + R_1, \;\mbox{and for $k \ge 1$}\,,
\\[1ex]
R_{k+1} & = R_1 \circ \theta_{D_k} + D_k, \;D_{k+1} = D_1 \circ \theta_{D_k} + D_k\,,
\end{split}
\end{equation}

\n
so that $0 \le R_1 \le D_1 \le \dots \le R_k \le D_k \le \dots \le \infty$. To control the dependence between the events in the right-hand side of (\ref{2.13}) we consider $r \ge 2$ and write:
\begin{equation}\label{2.23}
\mu^\prime_{1,2} + \mu^\prime_{2,1} = \dsl_{2 \le \ell \le r} \rho^\prime_\ell + \overline{\rho}\,,
\end{equation}
where
\begin{equation}\label{2.24}
\rho^\prime_\ell = 1_{\{R_\ell < \infty = R_{\ell + 1} \}} \;(\mu^\prime_{1,2} + \mu^\prime_{2,1}), \;\mbox{for $\ell \ge 1$}\,,
\end{equation}

\medskip\n
(note that $\rho^\prime_1 = 0$, since $(\mu^\prime_{1,2} + \mu^\prime_{2,1})$-a.s., $R_1 = 0$ and $R_2 < \infty$), and
\begin{equation}\label{2.25}
\ov{\rho} = 1_{\{R_{r+1} < \infty\}} (\mu^\prime_{1,2} + \mu^\prime_{2,1})\,,
\end{equation}

\medskip\n
($\ov{\rho}$ will be treated as a correction term, see (\ref{2.30}) and (\ref{2.61}) below). As as result of (\ref{2.20}) we see that
\begin{equation}\label{2.26}
\mbox{$\rho^\prime_\ell, 2 \le \ell \le r, \ov{\rho}, \mu^\prime_{1,1}, \mu^\prime_{2,2}$ are independent Poisson point processes on $W_+$}\,.
\end{equation}

\medskip\n
We denote with $\xi^\prime_\ell$ the intensity measure of $\rho^\prime_\ell$ and with $\ov{\xi}$ the intensity measure of $\ov{\rho}$. From the analogues of (\ref{2.16}) and (\ref{1.14}) with $u^\prime$ in place of $u$ we infer that
\begin{equation}\label{2.27}
\xi^\prime_\ell = u^\prime \,P_{e_V} [E, R_\ell < \infty = R_{\ell + 1}, \cdot ], \;2 \le \ell \le r\,,
\end{equation}
and that
\begin{equation}\label{2.28}
\ov{\xi} = u^\prime \,P_{e_V} [E, R_{r+1} < \infty, \,\cdot]\,,
\end{equation}

\medskip\n
where $E =\big\{X_0 \in \wt{C}_{m_1}, H_{\wt{C}_{m_2}} < \infty\} \cup \{X_0 \in \wt{C}_{m_2}, \,H_{\wt{C}_{m_1}} < \infty\big\}$.

\medskip
We first bound the total mass $\ov{\xi}(W_+)$ of $\ov{\xi}$. This is performed in an analogous fashion to (3.23)--(3.25) of \cite{Szni07a}. With (\ref{1.6}), (\ref{1.8}) and classical bounds on the Green function, cf.~\cite{Lawl91}, p.~31, we find that
\begin{equation}\label{2.29}
\sup\limits_{x \in U^c} \;P_x [H_V < \infty] \le \dis\frac{c_3}{\ell^{d-2}_n} \;.
\end{equation}

\medskip\n
Discarding the event $E$ in (\ref{2.28}) and applying the strong Markov property at times $D_r$, $D_{r-1}, \dots, D_1$, we find that
\begin{equation}\label{2.30}
\ov{\xi}(W_+) \le c\,u^\prime \,L_n^{d-2} \Big(\dis\frac{c_3}{\ell^{d-2}_n}\Big)^r \stackrel{(\ref{2.2})}{\le} c  \,c_3^r \;u^\prime \,L_n^{(d-2)(1-ar)}\,,
\end{equation}

\medskip\n
where we used the subadditivity property of the capacity, see below (\ref{1.5}), and the right-hand inequality of (\ref{1.8}) when bounding cap$(V)$, see (\ref{2.14}).

\medskip
We then turn our attention to the point measures $\rho^\prime_\ell$. We introduce the set of finite paths (see the beginning of Section 1 for the definition)
\begin{align}
\cT = \big\{ &\mbox{$w = \big(w(i)\big)_{0 \le i \le N}$ finite path; $w(0) \in V$, $w(N) \in \partial U$, and}\label{2.31}
\\
&\mbox{$w(i) \in U$, for $0 \le i < N\big\}$}\,.\nonumber
\end{align}

\medskip\n
Given $\ell \ge 1$, we define the map $\phi_\ell$ from $\{R_\ell < \infty = R_{\ell + 1}\} \subseteq W_+$ into $\cT^\ell$ such that:
\begin{align}
&w \rightarrow \phi_\ell(w) = (w_1,\dots,w_\ell), \;\mbox{with} \label{2.32}
\\
&w _k(\cdot) = \big(X_{R_k + \cdot} (w)\big)_{0 \le \cdot \le D_k(w) - R_k(w)}, \;\mbox{for} \;1 \le k \le \ell\,. \nonumber
\end{align}

\medskip\n
The next lemma yields a control on the probability that the walk hits $y \in V$ at its entrance in $V$ when it starts in $\partial U \cup \partial_{\rm int} U$.

\begin{lemma}\label{lem2.2}
\begin{equation}\label{2.33}
\sup\limits_{z \in \partial U \cup \partial_{\rm int} U} \,P_z[H_V < \infty, X_{H_V} = y] \le \dis\frac{c}{L^{d-2}_{n+1}} \,e_V(y), \;\mbox{for all $y \in V$}\,.
\end{equation}
\end{lemma}

\begin{proof}
From the inclusion $V \subseteq U$, one deduces the identity
\begin{equation}\label{2.34}
e_V(y) = P_{e_U} [H_V < \infty, \,X_{H_V}  = y], \;\mbox{for all $y \in V$}\,.
\end{equation}

\n
This ``sweeping'' identity can for instance be seen as the consequence of (1.46) of \cite{Szni07a}, for the intensity measures of the Poisson point processes under consideration there. Defining for $y \in V$ the non-negative harmonic function on $V^c$
\begin{equation*}
\psi(z) = P_z [H_V < \infty, \,X_{H_V} = y]\,,
\end{equation*}

\medskip\n
one finds with the left-hand inequality of (\ref{1.8}) and (\ref{2.34}) that
\begin{equation}\label{2.35}
e_V(y) \ge c\,L^{d-2}_{n+1} \;\inf\limits_{z \in \partial_{\rm int} U} \psi(z)\,.
\end{equation}

\n
With the Harnack inequality, cf.~\cite{Lawl91}, p.~42, and a covering argument of $\partial U \cup \partial_{\rm int} U$ by finitely many balls in $V^c$ with radius $c \,L_{n+1}$ and centers in $\partial_{\rm int} U$, one knows that
\begin{equation}\label{2.36}
\sup\limits_{z \in \partial U \cup \partial_{\rm int} U} \psi(z) \le c \,\inf\limits_{z \in \partial U \cup \partial_{\rm int} U} \psi(z)\,.
\end{equation}

\medskip\n
Inserting this inequality in (\ref{2.35}) yields for $y \in V$
\begin{equation*}
e_V(y) \ge c\,L_{n+1}^{d-2} \;\sup\limits_{z \in \partial U \cup \partial_{\rm int} U} P_z [H_V < \infty, X_{H_V} = y]\,,
\end{equation*}
and hence (\ref{2.33}).
\end{proof}

\n
For $2 \le \ell \le r$, we can view $\rho^\prime_\ell$ in (\ref{2.24}) as a point process on $\{R_\ell < \infty = R_{\ell + 1}\} ( \subseteq W_+)$ and then introduce
\begin{equation}\label{2.37}
\mbox{$\wt{\rho}\,^\prime_\ell$ the image of $\rho^\prime_\ell$ under $\phi_\ell$}\,.
\end{equation}
Thus with (\ref{2.26}) we find that
\begin{equation}\label{2.38}
\mbox{$\wt{\rho}\,^\prime_\ell$, $2 \le \ell \le r$, $\ov{\rho}$, $\mu^\prime_{1,1}$, $\mu^\prime_{2,2}$ are independent Poisson point processes}\,.
\end{equation}

\n
We denote with $\wt{\xi}\,^\prime_\ell$ the intensity of $\wt{\rho}\,^\prime_\ell$, for $2 \le \ell \le r$. In view of (\ref{2.27}) we see that for $w_1,\dots,w_\ell$ in $\cT$:
\begin{equation}\label{2.39}
\begin{array}{l}
\wt{\xi}\,^\prime_\ell(w_1,\dots,w_\ell) = 
\\[1ex]
u^\prime \,P_{e_V} \big[E, R_\ell < \infty = R_{\ell + 1}, (X_{R_k + \cdot})_{0 \le \cdot \le D_k - R_k} = w_k (\cdot), \,1 \le k \le \ell \big]\,.
\end{array}
\end{equation}

\n
We will now prove with the help of Lemma \ref{lem2.2} that for $2 \le \ell \le r$,
\begin{equation}\label{2.40}
\wt{\xi}\,^\prime_\ell \le c_4 \,u^\prime \,L^{d-2}_n \;\Big(\dis\frac{c_5}{\ell^{d-2}_n}\Big)^{\ell - 1} \,P_{\ov{e}_V} [(X_\point)_{0 \le \cdot \le T_U} \in \cdot]^{\otimes \ell}\,,
\end{equation}

\medskip\n
where $\ov{e}_V$ stands for the normalized equilibrium measure of $V$:
\begin{equation}\label{2.41}
\ov{e}_V = \dis\frac{1}{{\rm cap}(V)} \;e_V\,.
\end{equation}

\medskip\n
Indeed observe that for $w_1,\dots,w_\ell$ in $\cT$, in view of (\ref{2.39}), discarding the event $E$ and using the strong Markov property we find that
\begin{equation}\label{2.42}
\begin{array}{l}
\wt{\xi}\,^\prime_\ell (w_1,\dots,w_\ell) \le u^\prime \,E_{e_V}[R_\ell < \infty, (X_{R_k + \cdot})_{0 \le \cdot \le D_k - R_k} = w_k(\cdot), \;1 \le k < \ell\,,
\\[1ex]
P_{X_{R_\ell}} \big[(X_\point)_{0 \le \cdot \le T_U} = w_\ell(\cdot)]\big] = 
\\[1ex]
u^\prime\,E_{e_V} \big[D_{\ell-1}< \infty, (X_{R_k + \cdot})_{0 \le \cdot \le D_k - R_k} = w_k(\cdot), \;1 \le k < \ell\,,
\\
E_{X_{D_{\ell-1}}}\big[H_V < \infty, P_{X_{H_V}} [(X_\point)_{0 \le \cdot \le T_U} = w_\ell (\cdot)]\big]\big] \stackrel{(\ref{2.33})}{\le}
\\[1ex]
u^\prime\,P_{e_V} \big[R_{\ell - 1} < \infty, (X_{R_{k+ \cdot}})_{0 \le \cdot \le D_k - R_k} = w_k(\cdot), 1 \le k < \ell\big]
\\[1ex]
\dis\frac{c}{L_{n+1}^{d-2}} \,P_{e_V} \big[(X_\point)_{0 \le \cdot \le T_U} = w_\ell(\cdot)\big] \stackrel{\rm induction}{\le}
\\[1ex]
u^\prime \;\Big(\dis\frac{c}{L_{n+1}^{d-2}}\Big)^{\ell - 1} \;\mbox{\small $\prod\limits^\ell_{k=1}$} \;P_{e_V} \big[(X_\point)_{0 \le \cdot \le T_U} = w_k(\cdot)\big]\,.
\end{array}
\end{equation}

\medskip\n
The claim (\ref{2.40}) now follows by bounding cap$(V)$ with (\ref{1.8}) and the subadditive property below (\ref{1.5}), and using (\ref{2.41}).

\medskip
We can also view $1_{\{R_2 = \infty\}} (\mu^*_{1,1} + \mu^*_{2,2})$ as a Poisson point process on $\{R_1 < \infty = R_2\}$ $(\subseteq W_+)$, (note that $R_1 = 0$, $\mu^*_{1,1} + \mu^*_{2,2}$-a.s. in view of the analogue of (\ref{2.16}) for $\mu^*_{i,j}$). We then introduce the Poisson point process on $\cT$:
\begin{equation}\label{2.43}
\mbox{$\wt{\rho}\,^*_1$ the image of $1_{\{R_2 = \infty\}} (\mu^*_{1,1} + \mu^*_{2,2})$ under $\phi_1$}\,,
\end{equation}

\medskip\n
and denote with $\wt{\xi}\,^*_1$ its intensity measure (a measure on $\cT$). In view of (\ref{1.15}), (\ref{2.18}) and the analogue of (\ref{2.16}) for $\mu^*_{i,j}$, we find that:
\begin{equation}\label{2.44}
\wt{\xi}\,^*_1 (w) = (u-u^\prime) \,P_{e_V} [(X_\point)_{0 \le \cdot \le T_U} = w(\cdot), \;H_V \circ \theta_{T_U} = \infty], \;\mbox{for} \; w \in \cT\,.
\end{equation}

\n
As a result of (\ref{2.29}) we see that for $L_0 \ge c$,
\begin{equation}\label{2.45}
\inf\limits_{x \in \partial U} \,P_x[H_V = \infty] \ge \fr \;,
\end{equation}

\n
and hence with (\ref{1.8}) and the strong Markov property applied at time $T_U$ to the probability in (\ref{2.44}), we deduce that
\begin{equation}\label{2.46}
\wt{\xi}\,^*_1 \ge c_6(u-u^\prime) \,L_n^{d-2} \,P_{\ov{e}_V} [(X_\point)_{0 \le \cdot \le T_U} \in \cdot]\,.
\end{equation}

\medskip\n
The upper bounds (\ref{2.40}) and the lower bound (\ref{2.46}) will be our main instrument when seeking to dominate the trace on $V \cap \IZ^2$ of trajectories in the support of $\mu^\prime_{1,2} + \mu^\prime_{2,1}$ in terms of the trace on $\IZ^2 \cap V$ of trajectories in the support of $\mu^*_{1,1} + \mu^*_{2,2}$. With this goal in mind we introduce the random subsets of $V \cap \IZ^2$:
\begin{equation}\label{2.47}
\begin{split}
\cI^\prime_{i,i} &= V \cap \IZ^2 \cap \Big(\mbox{\f $\dis\bigcup\limits_{w \in {\rm Supp}(\mu^\prime_{i,i})}$} \;{\rm range}(w)\Big), \;\mbox{for $i = 1,2$}\,,
\\[1ex]
\wt{\cI}_{\ell}^{\prime} & = V \cap \IZ^2 \cap \Big(\mbox{\f $\dis\bigcup\limits_{(w_1,\dots,w_\ell) \in {\rm Supp} \,\wt{\rho}\,^\prime_\ell}$} \;{\rm range}(w_1) \cup \dots \cup {\rm range}(w_\ell)\Big), \;\mbox{for $2 \le \ell \le r$}\,,
\\[1ex]
\ov{\cI} & = V \cap \IZ^2 \cap \Big(\mbox{\f $\dis\bigcup\limits_{w \in {\rm Supp}\,\ov{\rho}}$}\,{\rm range}(w)\Big)\,.
\end{split}
\end{equation}

\medskip\n
Note that with (\ref{2.38}) it follows that
\begin{equation}\label{2.48}
\mbox{$\cI^\prime_{1,1}$, $\cI^\prime_{2,2}$, $\wt{\cI}\,^\prime_\ell$, $2 \le \ell \le r$, and $\ov{\cI}$ are independent under $\IP$}\,.
\end{equation}

\n
Moreover in view of (\ref{2.17}), (\ref{2.23}), (\ref{2.37}) one has the identity
\begin{equation}\label{2.49}
\cI^{u^\prime} \cap V \cap \IZ^2 = \cI^\prime_{1,1} \cup \cI^\prime_{2,2} \cup \Big(\mbox{\f $\dis\bigcup\limits_{2 \le \ell \le r}$} \,\wt{\cI}\,^\prime_\ell\Big) \cup \ov{\cI}\,.
\end{equation}
In a similar fashion to (\ref{2.47}) we can define 
\begin{equation}\label{2.50}
\begin{split}
\cI^*_{i,i} & = V \cap \IZ^2 \cap \Big(\mbox{\f $\dis\bigcup\limits_{w \in {\rm Supp} \,\mu^*_{i,i}}$} \,{\rm range}(w)\Big), \;\mbox{for $i = 1,2$}\,,
\\[1ex]
\cI^* & = V \cap \IZ^2 \cap \Big(\mbox{\f $\dis\bigcup\limits_{w \in {\rm Supp} \,\wt{\rho}^*_1}$} \;{\rm range}(w)\Big)\,.
\end{split}
\end{equation}

\n
Taking (\ref{2.20}) and (\ref{2.43}) into account we see that
\begin{equation}\label{2.51}
\mbox{$\cI^*$, $\cI^\prime_{1,1}$, $\cI^\prime_{2,2}$, $\wt{\cI}\,^\prime_\ell$, $2 \le \ell \le r$, $\ov{\cI}$ are independent under $\IP$}, 
\end{equation}
and further that
\begin{equation}\label{2.52}
\cI^* \subseteq \cI^*_{1,1} \cup \cI^*_{2,2}\,.
\end{equation}

\medskip\n
As we now explain we will construct a coupling of $\wt{\cI}^\prime_\ell$, $2 \le \ell \le r$, and $\cI^*$. We consider on some auxiliary probability space independent Poisson variables, $N^\prime_\ell$, $2 \le \ell \le r$, and $N^*_\ell$, $1 \le \ell \le r$, with respective intensities, cf.~(\ref{2.40}), (\ref{2.46})
\begin{equation}\label{2.53}
\begin{split}
\lambda^\prime_\ell & = c_4 \,u^\prime \,L_n^{d-2} \,\Big(\dis\frac{c_5}{\ell_n^{d-2}}\Big)^{\ell - 1}, \;2 \le \ell \le r, \;\mbox{and}
\\[1ex]
\lambda^*_\ell & = \dis\frac{c_6}{r} \;(u - u^\prime) \,L^{d-2}_n\,, \;\;1 \le \ell \le r\,,
\end{split}
\end{equation}

\n
as well as iid $\cT$-valued variables $\gamma^\ell_i$, $1 \le \ell \le r$, $i \ge 1$, independent from the $N^\prime_\ell$, $2 \le \ell \le r$, $N^*_\ell$, $1 \le \ell \le r$, with common distribution $P_{\ov{e}_V} [(X_\point)_{0 \le \cdot \le T_U} \in \cdot]$. We then define the point processes:
\begin{equation}\label{2.54}
\Gamma^\prime_\ell = \dsl_{1 \le i \le N^\prime_\ell} \;\delta_{(\gamma^\ell_{1 + (i-1)\ell}, \gamma^\ell_{2 + (i-1)\ell}, \dots, \gamma^\ell_{\ell + (i-1) \ell})}, \;\mbox{on $\cT^\ell$, with $2 \le \ell \le r$}\,,
\end{equation}
and
\begin{equation}\label{2.55}
\Gamma^*_\ell = \dsl_{1 \le i \le N_\ell^*} \delta_{\gamma^\ell_i}, \;\mbox{on $\cT$, with $1 \le \ell \le r$}\,.
\end{equation}
We thus find that
\begin{align}
&\mbox{$\Gamma^\prime_\ell$, $2 \le \ell \le r$, are independent Poisson point processes with respective} \label{2.56}
\\[-0.8ex]
&\mbox{intensity measures $\lambda^\prime_\ell \,P_{\ov{e}_V}[(X_\point)_{0 \le \cdot \le T_U} \in \cdot]^{\otimes \ell} \Big(\stackrel{(\ref{2.40})}{\ge} \wt{\xi}\,^\prime_\ell\Big)$, and} \nonumber
\\[2ex]
&\mbox{$\Gamma^*_\ell$, $1 \le \ell \le r$, are independent Poisson point processes with identical} \label{2.57}
\\[-1ex]
&\mbox{intensity measure $\lambda^*_\ell \,P_{\ov{e}_V}[(X_\point)_{0 \le \cdot \le T_U} \in \cdot] 
\Big( \stackrel{(\ref{2.46})}{\le} \mbox{\f $\dis\frac{1}{r}$}\;\wt{\xi}\,^*_1$}\Big) \,. \nonumber
\end{align}

\medskip\n
In view of (\ref{2.56}), (\ref{2.57}) we can thus construct on some probability space $(\Sigma, \cF, Q)$ a coupling of $\wt{\rho}\,^\prime_\ell$, $\wt{N}\,^\prime_\ell$, $\Gamma^\prime_\ell$, $2 \le \ell \le r$, $\wt{\rho}\,^*_1$, $N^*_\ell$, $\Gamma^*_\ell$, $1 \le \ell \le r$, so that
\begin{equation}\label{2.58}
\mbox{$\wt{\rho}\,^\prime_\ell \le \Gamma^\prime_\ell$, for $2 \le \ell \le r$, and $\dsl_{1 \le \ell \le r} \Gamma^*_\ell \le \wt{\rho}\,^*_1$}\,,
\end{equation}

\medskip\n
(for instance in view of the inequality in the last line of (\ref{2.56}) we construct the law of $\wt{\rho}\,^\prime_\ell$ by thinning $\Gamma^\prime_\ell$, and in view of the inequality in the last line of (\ref{2.57}) we construct the laws of $\wt{\rho}\,^*_1$ by adding an independent $\cT$-valued Poisson point process).

\medskip
From the definition of $\wt{\cI}\,^\prime_\ell$ in (\ref{2.47}) and $\cI^*$ in (\ref{2.50}) it then follows that on the event $\bigcap_{2 \le \ell \le r} \{N_\ell^* \ge r\,N^\prime_\ell\}$ one has
\begin{equation}\label{2.59}
\begin{array}{l}
\bigcup\limits_{2 \le \ell \le r} \wt{\cI}^\prime_\ell \stackrel{(\ref{2.54}),(\ref{2.58})}{\subseteq} 
\\
V \cap \IZ^2 \cap \Big(\mbox{\f $\dis\bigcup\limits_{2 \le \ell \le r}$} \Big(\mbox{\f $\dis\bigcup\limits_{1\le i \le N^\prime_\ell}$} {\rm range}(\gamma^\ell_{1 + (i-1)\ell}) \cup \dots \cup {\rm range}(\gamma^\ell_{\ell + (i-1)\ell})\Big)\Big) \subseteq
\\[1ex]
V \cap \IZ^2 \cap  \Big(\mbox{\f $\dis\bigcup\limits_{2 \le \ell \le r}$} \Big(\mbox{\f $\dis\bigcup\limits_{1\le j \le N^*_\ell}$} {\rm range}(\gamma^\ell_j)\Big)\Big) \stackrel{(\ref{2.55}),(\ref{2.58})}{\subseteq} 
\\[1.5ex]
V \cap \IZ^2  \cap \Big(\mbox{\f $\dis\bigcup\limits_{w \in {\rm Supp}\,\wt{\rho}\,^*_1}$} {\rm range}(w)\Big) = \cI^*.
\end{array}
\end{equation}

\n
Hence the coupling we constructed on $(\Sigma, \cF, Q)$ leads to the bound:
\begin{equation}\label{2.60}
Q\Big(\cI^* \supseteq \mbox{\f $\dis\bigcup\limits_{2 \le \ell \le r}$}  \wt{\cI}\,^\prime_\ell\Big) \ge 1 - \dsl_{2 \le \ell \le r} \,Q(N^*_\ell < r \,N^\prime_\ell)\,.
\end{equation}

\n
This inequality plays a pivotal role in the sprinkling technique we employ in order to control interactions. Indeed we can bound the probability in the right-hand side of (\ref{2.13}) as follows:
\begin{equation}\label{2.61}
\begin{array}{l}
\IP[B^{u^\prime}_{m_1} \cap B^{u^\prime}_{m_2}] \stackrel{(\ref{2.49})}{=} 
\\[1ex]
\IP\Big[\mbox{there are $*$-paths from $C_{m_1}$ to $\partial_{\rm int} \wt{C}_{m_1}$ and from $C_{m_2}$ to $\partial_{\rm int} \wt{C}_{m_2}$ in}
\\
\;\;\;\; \cI^\prime_{1,1} \cup \cI^\prime_{2,2} \cup \Big(\mbox{\f $\dis\bigcup\limits_{2 \le \ell \le r}$} \wt{\cI}\,^\prime_\ell\Big) \cup \ov{\cI}\Big] \stackrel{(\ref{2.51}), (\ref{2.60})}{\le}
\\[2ex]
\IP\big[\mbox{there are $*$-paths from $C_{m_1}$ to $\partial_{\rm int} \wt{C}_{m_1}$ and from $C_{m_2}$ to $\partial_{\rm int} \wt{C}_{m_2}$ in}
\\
\;\;\;\; \cI^\prime_{1,1} \cup \cI^\prime_{2,2} \cup \cI^* \cup \ov{\cI}\big] +  \dsl_{2 \le \ell \le r} Q(N^*_\ell < r\,N^\prime_\ell) \stackrel{(\ref{2.20}), (\ref{2.52})}{\le}
\\[2ex]
\IP\big[\mbox{there is a $*$-path from $C_{m_1}$ to $\partial_{\rm int} \wt{C}_{m_1}$ in $\cI^\prime_{1,1} \cup \cI^*_{1,1}\big]$}
\\[1ex]
\IP\big[\mbox{there is a $*$-path from $C_{m_2}$ to $\partial_{\rm int} \wt{C}_{m_2}$ in $\cI^\prime_{2,2} \cup \cI^*_{2,2} \big]$} \;+ 
\\[1ex]
\IP[\ov{\cI} \not= \phi] + \dsl_{2 \le \ell \le r} Q(N^*_\ell < r \,N^\prime_\ell) \stackrel{(\ref{2.8}), (\ref{2.19})}{\le} q_n(u)^2 + \ov{\xi}(W_+) + \dsl_{2 \le \ell \le r} Q(N^*_\ell < r \,N^\prime_\ell) ,
\end{array}
\end{equation}

\n
where we have used the fact that $(\cI^\prime_{i,i} \cup \cI^*_{i,i}) \cap \wt{C}_{m_j} = \phi$ for $1 \le i \not= j \le 2$, to decouple probabilities after the second inequality and bounded $\IP[\ov{\cI} \not= \phi]$ by $\IP[\ov{\rho} \not= 0] \le \ov{\xi}(W_+)$, in the last inequality.

\medskip
We will now bound the last term in the last line of (\ref{2.61}). To this effect we ensure that, see (\ref{2.53})
\begin{equation}\label{2.62}
\lambda^*_\ell \ge 4r \,\lambda^\prime_\ell, \;\mbox{for $2 \le \ell \le  r$} \,,
\end{equation}
by imposing $L_0 \ge c$ and
\begin{equation}\label{2.63}
u - u^\prime \ge c^\prime_0 \;\dis\frac{r^2}{\ell_n^{d-2}} \;u^\prime, \;\;\mbox{(with the choice (\ref{2.66}) below this will yield (\ref{2.11}))}.
\end{equation}

\medskip\n
As  result of (\ref{2.62}) we thus find that for $2 \le \ell \le r$
\begin{equation}\label{2.64}
\begin{split}
Q(N^*_\ell < r N^\prime_\ell) & \le Q\Big(N^*_\ell \le \mbox{\f $\dis\frac{\lambda^*_\ell}{2}$}\Big) + Q\Big(N^\prime_\ell \ge \mbox{\f $\dis\frac{\lambda^*_\ell}{2 r}$} \Big)
\\[1ex]
& \le c\;e^{-\frac{c}{r}\;\lambda_\ell^*} \stackrel{(\ref{2.53})}{=} c\;e^{-\frac{c}{r^2} (u - u^\prime) L_n^{d-2}}\,,
\end{split}
\end{equation}

\n
where we have dominated $N^\prime_\ell$ by a Poisson variable of intensity $\frac{\lambda^*_\ell}{4r}$ thanks to (\ref{2.62}), and used classical exponential bounds on the tail of Poisson variables in the second inequality.

\medskip
Coming back to (\ref{2.13}) we see that when $L_0 \ge c$, for $n \ge 0$, $r \ge 2$, when (\ref{2.63}) holds one finds collecting (\ref{2.30}), (\ref{2.61}), (\ref{2.64}) that
\begin{equation}\label{2.65}
q_{n + 1}\big(u^\prime) \le c \,\ell_n^2 \big(q_n (u)^2 + u^\prime \,c^r_3 \; L_n^{(d-2)(1-ar)} + r \,e^{-\frac{c}{r^2} (u - u^\prime)L^{d-2}_n}\big)\,.
\end{equation}
We can now choose
\begin{equation}\label{2.66}
r = \mbox{\f $\dis\frac{3}{a}$} \stackrel{(\ref{2.1})}{=} 300\,,
\end{equation}

\n
and with this choice (\ref{2.65}) is more than enough to yield the claim (\ref{2.12}).
\end{proof}

\begin{remark}\label{rem2.3} \rm  ~

\medskip\n
1) It is clear from the above proof that adjusting the choice of $r$ in (\ref{2.66}) we can produce an arbitrary negative power of $L_n$ in place of $L_n^{-2}$ in the right-hand side of (\ref{2.12}), (of course adjusting constants there as well). The present choice will suffice for our purpose. 

\medskip\noindent
2) The inequality (\ref{2.65}) and the auxiliary condition (\ref{2.63}) play a similar role to (3.52) and (3.45) of \cite{Szni07a}, (which pertain to the control of crossings of the vacant set and are later applied to increasing sequences $u_n$ which remain bounded).

\medskip
Let us give some comments. Roughly speaking, in \cite{Szni07a} one dominates the trace on the single box $\wt{C}_{m_2}$ of trajectories in $\mu^\prime_{1,2} + \mu^\prime_{2,1}$ in terms of traces on $\wt{C}_{m_2}$ of $\mu^*_{2,2}$. In the present work we must handle both boxes at once when bounding the probability in the right-hand side of (\ref{2.13}). We dominate the simultaneous traces on $\wt{C}_{m_1} \cap \IZ^2$ and $\wt{C}_{m_2} \cap \IZ^2$ of $\mu^\prime_{1,2} + \mu^\prime_{2,1}$ in terms of the independent traces of $\mu^*_{1,1}$ on $\wt{C}_{m_1} \cap \IZ^2$ and of  $\mu^*_{2,2}$ on $\wt{C}_{m_2} \cap \IZ^2$.

\medskip
Whereas in \cite{Szni07a} one specifies the parameter $r$ relatively late, depending on the successive choices of $L_0$ and $u_0$, in order to cope with the increasing sequences $u_n$ and the indeterminacy of $u_0$, in the present context $r$ is simply fixed by the choice (\ref{2.66}). However the last term of (\ref{2.65}) is specific to the present control of a Poisson point process affecting two boxes in terms of two independent Poisson point processes concerning each respective box. Quite naturally this term deteriorates when the intensity $\lambda^*_\ell$ in (\ref{2.53}) becomes small. 

\hfill $\square$
\end{remark}

We will now see how one can propagate controls on the probabilities $q_n(u_n)$, cf.~(\ref{2.8}), and choose sequences $u_n$ which decrease not too fast so that $u_\infty = \lim u_n > 0$, but still sufficiently fast so that $u_n - u_{n+1}$ is large enough. The last term in the right-hand side of (\ref{2.12}) suggests picking $u_{n+1}$ sufficiently smaller than $u_n$ in (\ref{2.67}) below, and not trying to saturate (\ref{2.11}), (with $u_{n+1}$, and $u_n$ respectively playing the roles of $u^\prime$ and $u$).

\medskip
Given $u_0$ in $(0,1]$ and $L_0 > 1$, we define the sequence $u_n, n \ge 0$, via
\begin{equation}\label{2.67}
u_{n+1} = \Big(1 + \dis\frac{1}{\log L_n}\Big)^{-1} \,u_n, \;\mbox{for $n \ge 0$}\,.
\end{equation}

\n
We now derive a crucial propagation of certain controls from one scale to the next. We refer to Proposition \ref{prop2.1} for notation.

\begin{proposition}\label{prop2.4} $(d \ge 3)$

\medskip
When $L_0 \ge c$ and $0 < u_0 \le 1$, if for some $n \ge 0$,
\begin{equation}\label{2.68}
\begin{array}{ll}
\hspace{-6cm} {\rm i)} & c_2 (u_n - u_{n+1}) \,L_n^{d-2} \ge 2 \log L_n\,,
\\[1ex]
\hspace{-6cm} {\rm ii)} & a_n \stackrel{\rm def}{=} c_1 \,\ell_n^2 \,q_n(u_n) \le L_n^{-1}\,,
\end{array}
\end{equation}

\medskip\n
then {\rm (\ref{2.68}) i)} and {\rm ii)} hold true with $n+1$ in place of $n$.
\end{proposition}

\begin{proof}
Observe that when $L_0 \ge c$,
\begin{equation*}
u_n \stackrel{(\ref{2.67})}{=} \Big(1 + \dis\frac{1}{\log L_n}\Big) \;u_{n+1} \ge \Big(1 + \dis\frac{c_0}{\ell^{d-2}_n}\Big) \,u_{n + 1}\,,
\end{equation*}

\n
and hence with (\ref{2.12}) where we set $u = u_n$ and $u^\prime = u_{n+1}$, noting that $u_{n+1} \le u_0 \le 1$, we find that:
\begin{equation}\label{2.69}
a_{n+1} \le \Big(\dis\frac{\ell_{n+1}}{\ell_n}\Big)^2 \,a^2_n + c(\ell_{n+1} \ell_n)^2 \,(L^{-2}_n + e^{-c_2(u_n - u_{n+1})L^{d-2}_n})\,.
\end{equation}

\n
We will now check (\ref{2.68}) i) at level $n+1$. For $L_0 \ge c$, we find that $1 \le \log L_k \le \log L_{k+1} \stackrel{(\ref{2.2})}{\le} (1 + a) \log L_k + c \le 2 \log L_k$, for all $k \ge 0$, and therefore when $L_0 \ge c^\prime$,
\begin{equation}\label{2.70}
\begin{array}{l}
c_2 (u_{n+1} - u_{n+2}) \,L^{d-2}_{n+1} \;\stackrel{(\ref{2.2}),(\ref{2.67})}{=} c_2 \;\dis\frac{u_{n+2}}{\log L_{n+1}} \;L_n^{d-2} \;\ell_n^{d-2} \ge
\\[1ex]
\dis\frac{c_2}{4} \;\dis\frac{u_{n+1}}{\log L_n} \;L_n^{d-2} \;\ell_n^{d-2} = c_2 (u_n - u_{n + 1}) \;L^{d-2}_n \; \dis\frac{\ell_n^{d-2}}{4} \stackrel{(\ref{2.68}) i)}{\ge} \fr  \;\log L_n \; \ell_n^{d-2} \ge
\\[2ex]
\frvier \;\log L_{n+1} \,\ell_n^{d-2} \ge 2 \log L_{n+1}\,.
\end{array}
\end{equation}

\medskip\n
This shows that (\ref{2.68}) i) holds at level $n+1$. We now check (\ref{2.68}) ii) at level $n+1$. We note that 
\begin{equation*}
\dis\frac{\ell_{n+1}}{\ell_n} \stackrel{(\ref{2.2})}{\le} c \;\dis\frac{L_n^{(a+1)a}}{L_n^a} = c\,L^{a^2}_n, \;\mbox{and} \;\ell_{n+1} \,\ell_n \le c\,L_n^{2a + a^2}\,.
\end{equation*}

\n
We thus see that when $L_0 \ge c$, (\ref{2.69}) and the induction hypothesis yield that
\begin{equation*}
\begin{array}{l}
a_{n+1} \le c\,L_n^{2a^2} \;a^2_n + c\,L_n^{4a + 2a^2} (L_n^{-2} + L^{-2}_n) \le c\,L_n^{4a + 2a^2 - 2} \stackrel{(\ref{2.1})}{\le} 
\\[1ex]
L^{-1}_{n+1} \;\dis\frac{L_{n+1}}{L_n} \;c\,L_n^{6a - 1} \;\stackrel{(\ref{2.2})}{\le} L^{-1}_{n+1} \;c\,L_n^{7a - 1} \le L^{-1}_{n+1}\,.
\end{array}
\end{equation*}

\n
This shows (\ref{2.68}) ii) at level $n+1$ and completes the proof of Proposition \ref{prop2.4}.
\end{proof}

With Proposition \ref{prop2.4} obtaining a control on the sequence of probabilities $q_n(u_n)$ is reduced to initiating the recurrence in (\ref{2.68}). We collect the results we will need for the next section in the following theorem, (see Proposition \ref{prop2.1} and (\ref{2.67}) for the notation).

\begin{theorem}\label{theo2.5}  $(d \ge 3)$

\medskip\n
When  $L_0 \ge c_7$, if for the choice
\begin{equation}\label{2.71}
u_0 = \dis\frac{4}{c_2} \;(\log L_0)^2 \,L_0^{-(d-2)}\,,
\end{equation}
it holds that
\begin{equation} \label{2.72}
c_1 \,\ell^2_0 \,q_0(u_0) \le L_0^{-1}\,,
\end{equation}
then
\begin{equation}\label{2.73}
c_1 \,\ell^2_n \,q_n (u_\infty) \le c_1 \,\ell^2_n \,q_n (u_n) \le L^{-1}_n, \;\mbox{for all $n \ge 0$}\,, 
\end{equation}

\medskip\n
where $u_\infty  \stackrel{\rm def}{=} u_0 \times \prod_{n \ge 0} \big(1 + \frac{1}{\log L_n}\big)^{-1} \in (0,1]$.
\end{theorem}

\begin{proof}
We see that for $L_0 \ge c_7$ with the choice in (\ref{2.71}), $0 < u_0 \le 1$, and also that:
\begin{equation*}
c_2 (u_0 - u_1) \,L_0^{d-2} = c_2 \;\dis\frac{u_1}{\log L_0} \;L_0^{d-2} \ge \dis\frac{c_2}{2} \;\dis\frac{u_0}{\log L_0} \;L_0^{d-2} = 2 \log L_0\,.
\end{equation*}

\medskip\n
The claim (\ref{2.73}) now follows from Proposition \ref{prop2.4} and the positivity of $u_\infty$ is a consequence of the inequality $L_n \ge L_0^{(1 + a)^n}$, cf.~(\ref{2.2}).
\end{proof}

\begin{remark}\label{rem2.6} \rm ~

\medskip\n
1)
It should be realized that initiating the induction scheme we have developed in this section, i.e. checking (\ref{2.72}) for $L_0 \ge c_7$ and $u_0$ as in (\ref{2.71}) is not a mere formality. To illustrate the point observe that one can replace $\IZ^2$ with $\IZ^d$ in (\ref{2.3}) and (\ref{2.7}), i.e. consider instead boxes of size of order $L_n$ filling the whole space $\IZ^d$ and the event that such a box is connected by a $*$-path in $\cI^u$ to the complement of the neighborhood of size of order $L_n$ corresponding to (\ref{2.5}) in this modified set-up. Then small variations of the arguments used in Propositions \ref{prop2.1}, \ref{prop2.4}, with the modified choice $a = \frac{1}{100d}$ in (\ref{2.1}), will yield a similar result as in Theorem \ref{theo2.5}, with possibly different constants and the multiplicative factor $\ell^2_n$ replaced with $\ell_n^{2(d-1)}$ in (\ref{2.73}), (this modification originates from the fact that a similar change takes place in (\ref{2.13}) and (\ref{2.65})). However one cannot initiate the induction scheme one obtains in this new set-up. Indeed the corresponding probabilities $q_n(u_n)$ tend to $1$ as $n$ goes to infinity, as can be seen from the fact that $\IP$-almost surely for large $n$ the box $C_m$ at level $n$ containing the origin meets $\cI^{u_\infty} \subseteq \cI^{u_n}$, and any trajectory in the interlacement entering $C_m$ also meets $\partial_{\rm int} \,\wt{C}_m$. 

\medskip
The above observation stresses the importance of being able to check the initial hypothesis of the induction and of the presence of $\IZ^2$ in the definitions (\ref{2.3}) and (\ref{2.7}). 

\bigskip\n
2)  Building up on Remark \ref{rem2.3} 1), the arguments employed in Proposition \ref{prop2.4} and Theorem \ref{theo2.5} also show with the appropriate adjustment of constants that given $M \ge 1$, when $L_0 \ge c(M)$, if for the choice $u_0 = c^\prime(M)(\log L_0)^2 \,L_0^{-(d-2)}$, one has $c^{\prime\prime}(M) \ell_0^2 \,q_0(u_0) \le L^{-M}_0$, then for all $n \ge 0$,
\begin{equation*}
c^{\prime\prime}(M) \ell_n^{2} \,q_n(u_\infty) \le c^{\prime\prime}(M) \ell_n^{2} \,q_n(u_n) \le L_n^{-M}\,,
\end{equation*}

\medskip\n
with $u_n$ as in (\ref{2.67}) and $u_\infty = u_0 \times \prod_{n \ge 0} (1 + \frac{1}{\log L_n})$.

\hfill $\square$
\end{remark}

\section{Percolation of the vacant set for small $u$}
\setcounter{equation}{0}

In this section we derive the main result of this article and show that for all $d \ge 3$, the vacant set at level $u$ percolates in $\IZ^2$ when $u$ is small enough, see Theorem \ref{theo3.4}. The main task is to initiate the induction scheme developed in the previous section. This is carried out in Theorem \ref{theo3.1} and relies on arguments which share a common flavor with some of the steps in the derivation of lower bounds for disconnection times of discrete cylinders with large boxes, cf.~Section 2 of \cite{DembSzni06} and Section 5 of \cite{Szni08}.

\medskip
In this section we consider as in (\ref{2.71}),
\begin{equation}\label{3.1}
L_0 > c_7 \;\mbox{and} \;u_0 = \dis\frac{4}{c_2} \;(\log L_0)^2 \,L_0^{-(d-2)}\,.
\end{equation}

\n
The following result is amply sufficient to show that we can find $L_0 > c_7$, such that (\ref{2.72}) holds true. We refer to (\ref{2.8}) for the notation.
\begin{theorem}\label{theo3.1} ($d \ge 3$)
\begin{equation}\label{3.2}
\lim\limits_{L_0 \r \infty} \,L_0^\rho \,q_0(u_0) = 0, \;\mbox{for all $\rho > 0$}\,.
\end{equation}
\end{theorem}

\begin{proof}
We introduce the event
\begin{align}
\cC_{L_0}  =& \;\mbox{there is a self-avoiding $*$-path in $ ([0,2 L_0) \times [0,6 L_0) \times \{0\}^{d-2}) \cap \cI^{u_0}$}  \label{3.3}
\\
&\; \mbox{starting in $\{0\} \times [0,6 L_0) \times \{0\}^{d-2}$ and ending in} \nonumber 
\\
&\; \{2L_0 - 1\} \times [0,6 L_0) \times \{0\}^{d-2}\,. \nonumber
\end{align}

\begin{center}
\psfragscanon
\includegraphics[width=6cm,height=6cm]{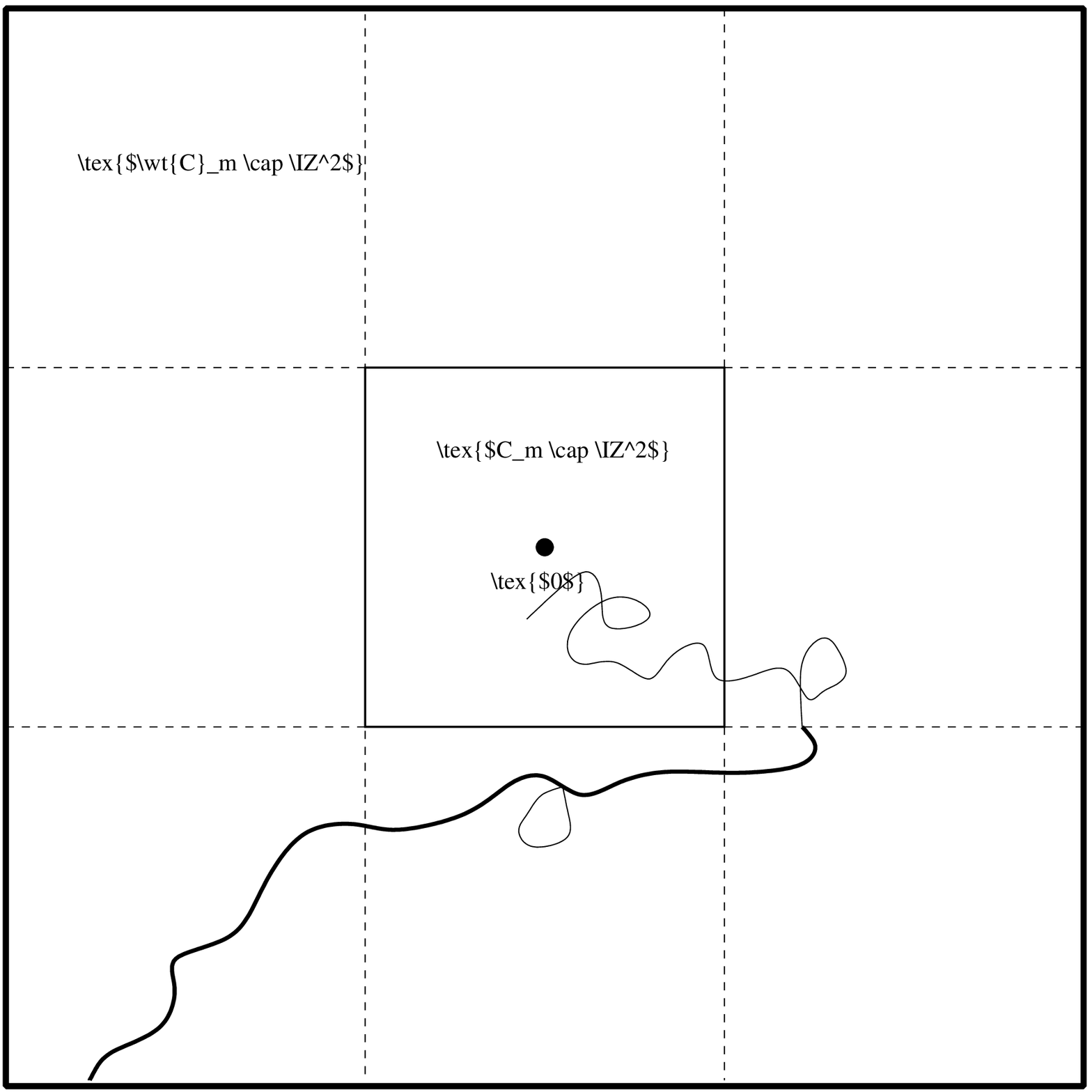}
\end{center}

\begin{center}
\begin{tabular}{ll}
Fig. 2: &An illustration of the fact that when $B_m^{u_0}$ occurs, the trace of $\cI^{u_0}$ on \\
&one of the four overlapping rectangles bordering $C_m \cap \IZ^2$ and isometric \\
&to $[0,2L_0) \times [0,6L_0) \times \{0\}^{d-2}$ contains a self-avoiding $*$-path between\\
&the long sides of the rectangle.
\end{tabular}
\end{center}

\medskip
Consider $m \in I_0$, cf.~(\ref{2.3}), the label of the box at level $0$ containing the origin. On the event $B_m^{u_0}$, cf.~(\ref{2.7}), we can find a $*$-self avoiding path from $\partial_{\rm int} \,\wt{C}_m$ to $\partial C_m$ in $(\wt{C}_m \backslash C_m) \cap \cI^{u_0} \cap \IZ^2$. One can extract from this path a self-avoiding $*$-path linking the long sides of one of the four overlapping rectangles isometric to $[0,2L_0) \times [0,6 L_0) \times \{0\}^{d-2}$ and covering $(\wt{C}_m \backslash C_m) \cap \IZ^2$. As a consequence of (\ref{1.14}) with $u = u_0$, or alternatively of the fact that the law of $\cI^{u_0}$ is invariant under the discrete isometries of $\IZ^d$, we find that
\begin{equation}\label{3.4}
q_0(u_0) \le 4 \IP [\cC_{L_0}]\,.
\end{equation}

\n
For large $L_0$ we then introduce the smaller length scale 
\begin{equation}\label{3.5}
L = 1000 [L_0 \,\exp\{- (\log L_0)^{\frac{1}{3}}\}]\,,
\end{equation}

\n
as well as the squares in the strip $U = [0,2 L_0) \times \IZ \times \{0\}^{d-2}$:
\begin{equation}\label{3.6}
\begin{array}{l}
C_{k,\ell} = k L \,e_1 + \ell \,\mbox{\f $\dis\frac{L}{1000}$} \;e_2 + [0,L)^2 \times \{0\}^{d-2} \subseteq U, \;\mbox{with}
\\[2ex]
0 \le k <  K = \Big[\mbox{\f $\dis\frac{2 L_0}{L}$}\Big] \;\mbox{and} \;|\ell | < K^\prime = \Big[ 10^4 \;\mbox{\f $\dis\frac{L_0}{L}$}\Big]\,.
\end{array}
\end{equation}

\medskip\n
For $k, \ell$ as above we also consider the box
\begin{equation}\label{3.7}
B_{k,\ell} = C_{k,\ell} \times [0,L)^{d-2} (\supseteq C_{k,\ell})\,,
\end{equation}

\n
and the nearly concentric sub-square of $C_{k,\ell}$:
\begin{equation}\label{3.8}
C^\prime_{k,\ell} = k L\,e_1 + \ell \;\mbox{\f $\dis\frac{L}{1000}$} \;e_2 + \Big[\mbox{\f $\dis \frac{4L}{10}, \frac{6L}{10}$}\Big]^2 \times \{0\}^{d-2}\,.
\end{equation}

\n
We denote with $\pi_1,\pi_2$ the projections from $\IZ^d$ onto $\IZ e_1$ and $\IZ e_2$ respectively. The next lemma is close in spirit to the geometric lemma of \cite{DembSzni06}, p.~332-334, albeit simpler due to the two-dimensional situation considered here.

\begin{lemma}\label{lem3.2}
For large $L_0$, on $\cC_{L_0}$, for each $0 \le k < K$ there exists $\ell_k$ with $|\ell_k| < K^\prime$, such that
\begin{equation}\label{3.9}
|\pi_1 \,(\cI^{u_0} \cap C_{k,\ell_k})| \ge \mbox{\f $\dis\frac{L}{4}$} \;\mbox{or}\; |\pi_2 (\cI^{u_0} \cap C_{k,\ell_k})| \ge \mbox{\f $\dis\frac{L}{4}$}\,.
\end{equation}
\end{lemma}

\begin{proof}
Consider $\o \in \cC_{L_0}$, we can find a self-avoiding $*$-path $x_i$, $0 \le i \le N$, as in (\ref{3.3}), such that in addition $|x_i - x_j|_\infty = 1$ implies $|i-j| = 1$, whenever $0 \le i$, $j \le N$, and only $x_0$ and $x_N$ belong to the sides of the strip $U$. Denote with $S = \{x_i, 0 \le i \le N\}$ the range of the path, and with $\cT op$ the connected component in $U \backslash S$ ``above the path'', i.e. containing $[0,2L_0) \times [6 L_0,\infty)$. With the Jordan curve theorem for polygons, cf.~\cite{Dies00}, p.~68, one sees that $[0,2 L_0) \times (-\infty,-1] \subseteq U \backslash S$, is disjoint from $\cT op$. When $L_0$ is large, it thus follows that
\begin{equation*}
\mbox{for $0 \le k < K$, $C^\prime_{k,\ell} \subseteq \cT op$, when $\ell = K^\prime - 1$ and $C^\prime_{k,\ell} \subseteq \cT op^c $, when $\ell = -(K^\prime - 1)$}\,.
\end{equation*}

\n
Tracking the relative fraction of points of $\cT op$ in $C^\prime_{k,\ell}$ as $\ell$ varies, one sees ``by continuity'', that for some $|\ell_k| < K^\prime$,
\begin{equation*}
C^\prime_{k,\ell_k} \cap \cT op \not= \phi \;\mbox{and} \;C^\prime_{k,\ell_k} \cap (\cT op)^c \not= \phi\,.
\end{equation*}

\n
This implies that $C^\prime_{k,\ell_k}$ meets $S$ as well. By definition of $S$, we can then find a $*$-path from $C^\prime_{k,\ell_k}$ to $\partial_{\rm int} C_{k,\ell_k}$. 
Either the $\pi_1$ or $\pi_2$-projection of this path contains at least $\frac{L}{4}$ points and the claim (\ref{3.9}) follows. 
\end{proof}

\begin{center}
\psfragscanon
\includegraphics[width=8cm,height=8cm]{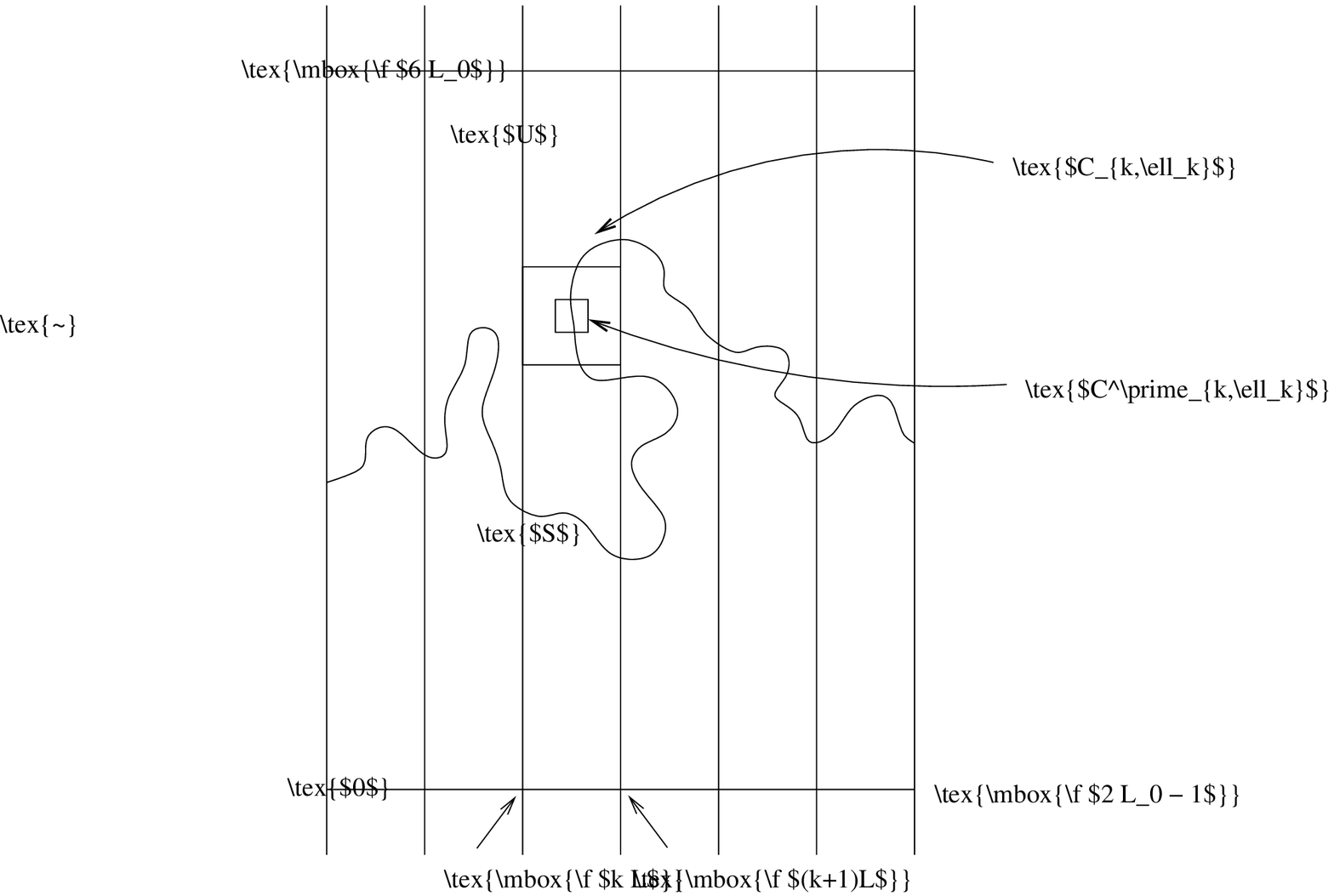}
\end{center}

\begin{center}
Fig. 3: An illustration of the boxes $C_{k,\ell_k}$ and $C^\prime_{k,\ell_k}$. 
\end{center}

\medskip
As a result of the above lemma we see that for large $L_0$,
\begin{equation}\label{3.10}
\cC_{L_0} \subseteq \bigcup\limits_{(\ell_k)} \;\cA_1^{(\ell_k)} \cup \cA_2^{(\ell_k)}\,,
\end{equation}

\n
where $(\ell_k)$ runs over all maps $k \in [0,K) \rightarrow \ell_k \in (-K^\prime,K^\prime)$, with $K, K^\prime$ defined in (\ref{3.6}), and for each $(\ell_k)$ we have set:
\begin{equation}\label{3.11}
\cA_i^{(\ell_k)} = \Big\{\o \in \Omega; \;\Big |\Big\{r \in  [0,K); \;|\pi_i (\cI^{u_0} \cap C_{r,\ell_r}) | \ge \mbox{\f $\dis\frac{L}{4}$}\Big\}\Big| \ge \mbox{\f $\dis\frac{K}{2}$} \Big\}, \;\mbox{for $i = 1,2$}\,.
\end{equation}

\n
With a rough counting argument we thus find that for large $L_0$
\begin{equation}\label{3.12}
\IP[\cC_{L_0}] \le e^{c \,\frac{L_0}{L} \,\log\big(c\,\frac{L_0}{L}\big)} \;\sup\limits_{(\ell_k), i = 1,2} \IP[\cA_i^{(\ell_k)}]\,,
\end{equation}

\medskip\n
where in the above supremum $(\ell_k)$ runs over the same collection as in (\ref{3.10}).

\medskip
We will now bound $\IP[\cA_1^{(\ell_k)}]$ uniformly in $(\ell_k)$. An analogous bound for $\IP[\cA_2^{(\ell_k)}]$ is derived in a similar fashion. We thus consider some $(\ell_k)$ as above and for $0 \le k < K$, denote with $\cS_k$ the collection of ``vertical segments'' $\cI$ of the form $\pi_1^{-1}(v) \cap C_{k,\ell_k}$, for $v \in [k L, (k+1) L) = \pi_1 (C_{k,\ell_k})$. We also write $\cS = \bigcup_{0 \le k < K} \cS_k$ for the collection of vertical segments in all $C_{k,\ell_k}$, $0 \le k < K$. The next step is, (see (\ref{1.2}) for the  notation):
\begin{lemma}\label{lem3.3}
For large $L_0$, one has
\begin{equation}\label{3.13}
\sup\limits_{x \in \IZ^d} \;E_x \Big[\exp\Big\{ \dis\frac{c_8}{(\log \frac{L_0}{L})}\;\mbox{\f $\dis\frac{\log L}{L}$} \;\dsl_{I \in \cS} \,1_{\{H_I < \infty\}}\Big\}\Big] \le 2\,.
\end{equation}
\end{lemma}

\begin{proof}
With Khashminskii's lemma, cf.~\cite{Khas59} as well as (\ref{2.46}) of \cite{DembSzni06}, it suffices to prove that for large $L_0$,
\begin{equation}\label{3.14}
\sup\limits_{x \in \IZ^d} \;E_x \Big[\dsl_{I \in \cS} \;1_{\{H_I < \infty\}}\Big] \le c\,\log \Big(\mbox{\f $\dis\frac{L_0}{L}$}\Big) \;\mbox{\f $\dis\frac{L}{\log L}$} \;.
\end{equation}

\medskip\n
With the help of (\ref{1.7}) and classical bounds on the Green function, cf. \cite{Lawl91}, p.~31, we see that for any $0 \le k < K$, and $x \in C_{k,\ell_k}$ one has
\begin{equation*}
\begin{split}
\dsl_{I \in \cS_k} \,P_x[H_I < \infty] & \le \dsl^L_{\ell = 1} \;\dis\frac{c}{\ell^{d-3}}\,, \;\;\mbox{if $d \ge 4$}\,,
\\
& \le \dsl^L_{\ell = 1} \;c\; \dis\frac{1 + \log(\frac{L}{\ell})}{\log L}\;, \;\;\mbox{if $d = 3$}\,.
\end{split}
\end{equation*}
Hence we see that
\begin{equation}\label{3.15}
\sup\limits_{x \in C_{k,\ell_k}} E_x \Big[\dsl_{I \in \cS_k} 1_{\{H_I < \infty\}}\Big] \le c\;\dis\frac{L}{\log L}, \;\mbox{for $0 \le k < K$}\,.
\end{equation}

\medskip\n
Applying once again standard bounds on the Green function and (\ref{1.7}) we find that for $x \in C \stackrel{\rm def}{=} \bigcup_{0 \le k < K} C_{k,\ell_k}$, one has
\begin{equation}\label{3.16}
\dsl_{0 \le k < K} P_x[H_{C_{k,\ell_k}} < \infty] \stackrel{(\ref{3.7})}{\le} \dsl_{0 \le k < K} P_x [H_{B_{k,\ell_k}} < \infty] \le c \dsl_{1 \le j \le [\frac{L_0}{L}]} \;\dis\frac{L^{d-2}}{(jL)^{d-2}} \le c \log \mbox{\f $\dis\frac{L_0}{L}$}\,.
\end{equation}

\n
Thus for any $x \in \IZ^d$ applying the strong Markov property we find that:
\begin{equation}\label{3.17}
\begin{array}{l}
E_x \Big[\dsl_{I \in \cS} 1_{\{H_I < \infty\}}\Big] = E_x \Big[H_C < \infty, \;E_{X_{H_C}} \Big[ \dsl_{I \in \cS} \,1_{\{H_I < \infty\}}\Big]\Big] \le
\\[2ex]
\sup\limits_{z \in C} \,E_z \Big[\dsl_{0 \le k < K} \;\dsl_{I \in \cS_k} 1_{\{H_I < \infty\}}\Big] \le
\\[1ex]
 \sup\limits_{z \in C} \;\dsl_{0 \le k < K} \;E_z \Big[H_{C_{k,\ell_k}} < \infty, E_{X_{H_{C_{k,\ell_k}}}} \Big[\dsl_{I \in \cS_k} 1_{\{H_I < \infty\}}\Big]\Big] \stackrel{(\ref{3.15}),(\ref{3.16})}{\le}
 \\
 \\[-2ex]
c\;\mbox{\f $\dis\frac{L}{\log L}$} \;\log \Big(\mbox{\f $\dis\frac{L_0}{L}$}\Big) \;.
\end{array}
\end{equation}

\medskip\n
This proves (\ref{3.14}) and thus concludes the proof of the lemma.
\end{proof}

We now resume the task of bounding $\IP[\cA_1^{(\ell_k)}]$. We introduce the non-negative measurable function on $W_+$:
\begin{equation}\label{3.18}
\phi(w) = \dis\frac{c_8}{\log\big(\frac{L_0}{L}\big)} \;\mbox{\f $\dis\frac{\log L}{L}$} \;\dsl_{I \in \cS} \;1_{\{H_I (w) < \infty\}}\,,
\end{equation}

\medskip\n
as well as $B = B(0,20 L_0) \supseteq C$, for large $L_0$, (see above (\ref{3.16}) for the notation). Hence for large $L_0$, we find that in the notation of (\ref{1.13})
\begin{equation}\label{3.19}
\begin{array}{l}
\IE[\exp\{ < \mu_{B, u_0}, \phi > \}] \stackrel{(\ref{1.14})}{=} \exp\{u_0 \,E_{e_B} [e^\phi - 1]\} \stackrel{(\ref{3.13})}{\le}   
\\[1ex]
\exp\{c \, u_0 \,{\rm cap}(B)\} \stackrel{(\ref{1.8}),(\ref{3.1})}{\le} \exp\{c (\log L_0)^2\}\,.
\end{array}
\end{equation}

\medskip\n
Since on $\cA_1^{(\ell_k)}$ one has $\langle \mu_{B,u_0}, \phi \rangle \ge \frac{c}{\log(\frac{L_0}{L})} \;\frac{\log L}{L} \;L_0$, it follows that
\begin{equation*}
\IP[\cA_1^{(\ell_k)}] \le \exp\Big\{- \dis\frac{c}{\log \big(\frac{L_0}{L}\big)} \;\;\mbox{\f $\dis\frac{\log L}{L}$} \;L_0 + c^\prime (\log L_0)^2\Big\} \,.
\end{equation*}

\n
With (\ref{3.5}) we also see that for large $L_0$,
\begin{equation}\label{3.20}
\begin{split}
c\;\log L_0 & \le \log L \le c^\prime \log L_0, \;c \exp\{{(\log L_0)^\frac{1}{3}}\} \le \mbox{\f $\dis\frac{L_0}{L}$} \le c^\prime \exp\{(\log L_0 )^{\frac{1}{3}}\}, \;\mbox{and}
\\
\log\Big(\mbox{\f $\dis\frac{L_0}{L}$}\Big)& \le c(\log L_0)^{\frac{1}{3}}\;.
\end{split}
\end{equation}

\n
As a result we deduce that for large $L_0$
\begin{equation}\label{3.21}
\sup\limits_{(\ell_k)} \IP[\cA_1^{(\ell_k)}] \le \exp\big\{ - c (\log L_0)^{\frac{2}{3}} \;e^{(\log L_0)^{\frac{1}{3}}}\big\}\,.
\end{equation}

\n
A similar bound holds with $\cA_2^{(\ell_k)}$ in place of $\cA_1^{(\ell_k)}$. Coming back to (\ref{3.12}) we then see with (\ref{3.20}) that for large $L_0$
\begin{equation}\label{3.22}
\IP[\cC_{L_0} ] \le \exp\{ - c (\log L_0)^{\frac{2}{3}} \;e^{(\log L_0)^{\frac{1}{3}}}\big\}\,,
\end{equation}

\medskip\n
which is more than enough to prove the claim (\ref{3.2}).
\end{proof}

We now come to the main result showing that for small $u$ the vacant set at level $u$ percolates in planes.

\begin{theorem}\label{theo3.4} $(d \ge 3)$

\medskip
For small $u > 0$,
\begin{equation}\label{3.23}
\mbox{$\IP$-a.s., $\cV^u \cap \IZ^2$ contains an infinite connected component}\,.
\end{equation}
\end{theorem}

\begin{proof}
The argument is in essence the same as for the proof of Theorem 4.3 of \cite{Szni07a}. With Theorem \ref{theo3.1} we can pick $L_0 \ge c_7$ such that (\ref{2.72}) holds and conclude with Theorem \ref{theo2.5} that
\begin{equation}\label{3.24}
c_1 \,\ell_n^2 \,q_n (u_\infty) \le L_n^{-1}, \;\mbox{for all $n \ge 0$, with $u_\infty \in (0,1]$}\,.
\end{equation}

\n
Then for $n _0 \ge 0$  and $M = 2 L_{n_0} -1$ we can write for $u \le u_\infty$
\begin{equation}\label{3.25}
\begin{array}{l}
\mbox{$\IP[0$ does not belong to an infinite connected component of $\cV^u \cap \IZ^2] \le$}
\\[1ex]
\mbox{$\IP[\cI^u \cap B(0,M) \cap \IZ^2 \not= \phi] + \IP[\cI^u \cap (\IZ^2 \backslash B(0,M))$ contains a $*$-circuit}
\\[1ex]
\mbox{surrounding $0]$}\,.
\end{array}
\end{equation}

\n
With (\ref{0.1}), (see also (1.58) of \cite{Szni07a}), $\IP[x \in \cV^u] = e^{-\frac{u}{g(0)}}$, with the notation as below (\ref{1.3}). The above expression is thus smaller than 
\begin{equation*}
\begin{array}{l}
cM^2 ( 1 - e^{-\frac{u}{g(0)}}) + \sum_{n \ge n_0} \IP[\cI^u \cap (\IZ^2 \backslash B(0,M))\;\mbox{contains a $*$-circuit}
\\
\mbox{surrounding $0$ and passing through a point in $[2 L_n, 2 L_{n+1} - 1] e_1]$}.
\end{array}
\end{equation*}

\medskip\n
We can cover $[2 L_n, 2L_{n+1} - 1] \,e_1$ with the translates of the box $C_m$ at level $n$ containing the origin by the vectors $(2 k + 1) \,L_n \,e_1$, $1 \le k \le \ell_n - 1$. A $*$-circuit in $\cI^u \cap (\IZ^2 \backslash B(0,M))$ surrounding $0$ and passing through $(2k + 1)\,L_n \,e_1 + C_m$ necessarily meets $(2 k + 1) L_n \,e_1 + \partial_{\rm int} \wt{C}_m$. As a result the left-hand side of (\ref{3.25}) is smaller than
\begin{equation*}
c \,M^2 (1 - e^{-\frac{u}{g(0)}}) + \dsl_{n \ge n_0} \;\ell_n\,q_n(u) \stackrel{(\ref{3.24})}{\le} c\Big(L^2_{n_0} u + \dsl_{n \ge n_0} \,L_n^{-1}\Big) < 1\,,
\end{equation*}

\n
if we choose $n_0$ large and $u \le c(L_0, n_0)$. This yields a positive $u$ such that
\begin{equation}\label{3.26}
\mbox{$\IP[0$ belongs to an infinite component of $\cV^u \cap \IZ^2] > 0$}\,.
\end{equation}

\n
However (\ref{2.6}) of \cite{Szni07a} shows that the law on $\{0,1\}^{\IZ^2}$ of the indicator function of $\cV^u \cap \IZ^2$ is ergodic under translations. With (\ref{3.26}) we see that the translation invariant event on $\{0,1\}^{\IZ^2}$ consisting of configurations for which the set of locations where the configuration takes the value $1$ contains an infinite connected component, has full measure under this law. The claim (\ref{3.23}) follows.
\end{proof}

\begin{remark}\label{rem3.5} \rm  ~

\medskip\n
1) As mentioned in the Introduction the above theorem combined with the results of \cite{Szni07a} completes the proof of the non-degeneracy for all $d \ge 3$ of the critical parameter $u_*$ of (\ref{0.3}). It even proves the non-degeneracy for all $d \ge 3$ of the critical parameter attached (with a straightforward modification of (\ref{0.2}), (\ref{0.3})) to the percolation of $\cV^u \cap \IZ^2$. It is plausible that this critical value is strictly smaller than $u_*$. However one may wonder whether it is possible to approximate $u_*$ from below by critical values corresponding to percolation of the vacant set in thick two-dimensional slabs, as in the case of Bernoulli percolation, see \cite{Grim99}, p.~148. 

\bigskip\n
2) With Remark \ref{rem2.6} 2) and a small variation in the proof of Theorem \ref{theo3.1}, (cf.~(\ref{3.19}) where the definition of $u_0$ involves a new constant $c^\prime(M)$ in place of $4/c_2$), we see that for any $M \ge 1$, when $L_0 \ge c(M)$, then
\begin{equation}\label{3.27}
c^{\prime\prime}(M) \ell_n^{2} \,q_n(u_\infty) \le c^{\prime\prime} (M) \ell_n^{2} \,q_n(u_n) \le L_n^{-M}, \;\mbox{for all $n \ge 0$}\,,
\end{equation}

\medskip\n
with $u_n$ as in (\ref{2.67}) and $u_\infty  = u_0 \times \prod_{n \ge 0} (1 + \frac{1}{\log L_n})$.

\medskip
As a direct consequence of this result, (see also (\ref{2.1}), (\ref{2.2}), (\ref{2.8})), it follows that for any $\rho > 0$, there exists $u(\rho) > 0$, such that for $u \le u(\rho)$,
\begin{equation}\label{3.28}
\lim\limits_{L \r \infty} \,L^\rho \,\IP[\mbox{there is a $*$-path from $0$ to $S(0,L)$ in $\cI^u \cap \IZ^2] = 0$}\,.
\end{equation}
\hfill $\square$
\end{remark}

\end{document}